\documentclass{amsart}
\usepackage{amsfonts, graphicx, amsmath, amssymb, color, verbatim}
\usepackage{pdfsync}
\usepackage[all, 2cell,dvips]{xypic}
\usepackage[margin = 1.5in]{geometry}

\newcommand{\p}{\partial}

\newcommand \lra {\longrightarrow}

\newcommand \absv [1]{\left \lvert #1 \right \rvert }
\newcommand \lp {\left(}
\newcommand \rp {\right)}
\newcommand \la {\langle}
\newcommand \ra {\rangle}

\newcommand{\wt}[1]{\widetilde{#1}}
\newcommand{\norm}[2][]{\left \| #2 \right \|_{#1} }

\newcommand\RR{\mathbb{R}}
\newcommand\NN{\mathbb{N}}
\newcommand\ZZ{\mathbb{Z}}
\newcommand\sphvar{y}
\newcommand\dspharv{\eta}
\newcommand\Id{\operatorname{Id}}

\newcommand\scl {\mathrm{scl}}

\newcommand\Pscl{\Psi_{\scl}}
\newcommand\X{\mathcal{X}}
\newcommand\Z{\mathcal{Z}}
\renewcommand\Vec{\mathcal{V}}

\newcommand\SR{\operatorname{SR}}
\newcommand\trho{{\rho_{\tilde Z}}}

\newcommand\Trace{\operatorname{Trace}}
\newcommand\tV{\tilde V}
\newcommand\Zsm{{Z_{\mathrm{sm}}}}
\newcommand\Fsm{{H_{\mathrm{sm}}}}
\newcommand\FFF{H}
\newcommand\ZZZ{\tilde Z}
\newcommand\Cinfe{C^{\infty, \epsilon}}
\newcommand\Cinfep{C^{\infty, \underline{\epsilon}}}
\newcommand\blambda{{\boldsymbol\lambda}}

\newtheorem{theorem}{Theorem}
\newtheorem{lemma}[theorem]{Lemma}
\newtheorem{proposition}[theorem]{Proposition}
\newtheorem{corollary}[theorem]{Corollary}
\newtheorem{definition}[theorem]{Definition}

\theoremstyle{remark}
\newtheorem{remark}[theorem]{Remark}

\numberwithin{equation}{section}
\numberwithin{theorem}{section}

\DeclareMathOperator \Op {Op}
\DeclareMathOperator \WF {WF}

\DeclareMathOperator \spec {spec}
\DeclareMathOperator \supp {supp}
\DeclareMathOperator \Tr {Tr}
\DeclareMathOperator \tr {Tr}

\DeclareMathOperator \Vol {Vol}

\DeclareMathOperator \sgn {sgn}

\title[Distribution of phase shifts for potentials
  with polynomial decay]{The distribution of phase shifts for semiclassical potentials
  with polynomial decay}
\author{Jesse Gell-Redman}
\address{Department of Mathematics, Johns Hopkins University}
\email{jgell@math.jhu.edu}
\author{Andrew Hassell}
\address{Mathematical Sciences Institute, Australian National University}
\email{Andrew.Hassell@anu.edu.au}

\begin{document}

\maketitle

\begin{abstract}
 This is the third paper in a series \cite{DGHH2013,
   Gell-Redman-Hassell-Zelditch} analyzing the asymptotic distribution
 of the phase shifts in the semiclassical limit.
We analyze the distribution of phase shifts, or equivalently, 
eigenvalues of the scattering matrix, $S_h(E)$, for semiclassical
Schr\"odinger operators on $\mathbb{R}^d$ which are perturbations of
the free Hamiltonian by a potential $V$ with polynomial decay.
Our assumption is that $V(x) \sim |x|^{-\alpha} 
v(\hat x)$ as $x \to \infty$, for some $\alpha > d$, with corresponding derivative estimates.  
In the semiclassical limit $h \to 0$, we show that the atomic measure on the 
unit circle defined by these eigenvalues, 
after suitable scaling in $h$, tends to a measure $\mu$ on $\mathbb{S}^1$.
Moreover, $\mu$ is the pushforward  from
$\mathbb{R}$ to $\mathbb{R} / 2 \pi
\mathbb{Z} = \mathbb{S}^1$ of a homogeneous distribution $\nu$ of
order $\beta$ depending on the dimension $d$ and the rate of decay $\alpha$ of the
potential function.  As a corollary we obtain an asymptotic formula
for the accumulation of phase shifts in a sector of $\mathbb{S}^1$.

The proof relies on an extension of results in \cite{HW2008} on the
classical Hamiltonian dynamics and semiclassical Poisson operator
to the class of potentials under consideration here.
\end{abstract}

\section{Introduction}
Consider a semiclassical Schr\"odinger operator 
$$
H_h := h^2 \Delta + V - E
$$
on $\mathbb{R}^d$, where $\Delta = - \sum_{i = 1}^d \p_{x_i}^2$ is the
positive Laplacian, $E$ is a real constant and $V \colon \mathbb{R}^d \lra \mathbb{R}$ is a
smooth real-valued function satisfying 
\begin{equation}
  \label{eq:potential-asymptotics}
  V(x) = \frac{v_0(\hat x)}{|x|^\alpha} + W(x), \quad  x \in \RR^d, \quad \hat x = \frac{x}{|x|} , \quad \mbox{ where } 
  \big| W(x) \big| = O(|x|^{-(\alpha + \epsilon)}) , 
\end{equation}
for $|x|$ large, some $\alpha > 1$ and some $\epsilon > 0$. (For our
main theorem, we will require $\alpha > d$, and $W$ to satisfy `symbolic' derivative estimates as in   \eqref{W-deriv-est}, but
for some of our intermediate results $\alpha > 1$ and \eqref{eq:potential-asymptotics} will be sufficient.)
 
 Under these assumptions, the (relative) scattering matrix $S_h$
exists and is a unitary operator on $L^2(\mathbb{S}^{d - 1})$; $S_h$ is given on $\phi \in C^\infty(\mathbb{S}^{d
  -1 })$ as $S_h \phi = e^{i\pi(d -1)/2} \psi$ where $\psi$ is the
unique function such that there is a
solution $u_\phi$ to $H_h u_\phi = 0$ satisfying
\begin{equation}
u_\phi = r^{-(d - 1)/2} (e^{- i \sqrt{E} r/h} \phi(\omega) + e^{i
  \sqrt{E} r/h}\psi(-\omega) ) + o(r^{(d - 1)/2}), \quad r = |x|. \label{eq:generalized-eigenfunction}
\end{equation}
The difference $S_h - \Id$ is a compact operator on $L^2(\mathbb{S}^{d - 1})$, and thus
the spectrum of $S_h$ lies on the unit circle, is discrete, and accumulates
only at $1$.  Setting $\gamma = (d - 1)/(\alpha - 1)$, we define the
(infinite) atomic measure $\mu_h$ on the circle
which acts on $f \in C^0_{comp}(\mathbb{S}^1 \setminus 1)$ by
\begin{equation}
  \label{eq:measure-definition}
  \la \mu_h, f \ra = h^{\gamma \alpha} \sum_{e^{2i \beta_{n, h} } \in \spec(S_h)}
  f(e^{2i \beta_{n, h}})
\end{equation}
for some enumeration $e^{2i \beta_{n, h}}$ of the eigenvalues of $S_h$, repeated according to their multiplicity. 
  The $\beta_{n, h} \in [0, \pi)$ are called  the `phase shifts' of $H_h$.

\medskip

\noindent \textbf{Main Theorem:}
\textit{Let $I$ be any open interval of $\mathbb{S}^1$ containing $1$. 
Assume that $V$ satisfies \eqref{eq:potential-asymptotics} for some
$\alpha > d$, and that $W$ in \eqref{eq:potential-asymptotics}
satisfies the additional derivative estimates 
\begin{equation}
\big| \p_x^k W(x) \big| = O(|x|^{-(\alpha  + |k| + \epsilon)}) \ \forall \ k \in \NN^d , \quad |x| \to \infty. 
\label{W-deriv-est}\end{equation}
Then the measures $\mu_h$ converge in the weak-$*$ topology on $\mathbb{S}^1 \setminus I$
to a measure $\mu$. Moreover, $\mu$ is the pushforward via the map $\mathbb{R} \lra
\mathbb{R}/ 2\pi \mathbb{Z} \simeq \mathbb{S}^1$ of a homogeneous
measure 
$$
\nu = \left\{
  \begin{array}{lr}
    (1/(2\pi))^{d-1}a_1 \theta^{- \gamma + 1}  & \mbox{ for } \theta > 0\\ 
    (1/(2\pi))^{d-1}a_2 |\theta|^{- \gamma + 1} & \mbox{ for } \theta < 0.
  \end{array}
\right.
$$
Concretely, for every $f \in C^0_{comp}(\mathbb{S}^1 \setminus I)$, we have 
$$
\lim_{h \to 0}  \la \mu_h, f \ra = \int_{\mathbb{S}^1} f(e^{i\theta}) \, d\mu. 
$$
The constants $a_1, a_2$ are given in
\eqref{eq:constants-3} below.  
 }
\medskip

The Main Theorem is proven at the end of Section
\ref{sec:main-theorem}, modulo the proofs of subsequent technical lemmas.

It follows from the Main Theorem that eigenvalues of $S_h$ accumulate
in sectors of the unit circle at a rate of $h^{- \alpha( d -1)/(\alpha
  - 1)}$.  Indeed, defining a sector on the circle by choosing angles
  $0 < \phi_0 < \phi_1 < 2\pi$, and letting
$$
N(\phi_0, \phi_1) = \# \{ n : \phi_0 \le \beta_{h, n} \le \phi_1
 \mbox{ mod }  2 \pi \},
$$
the main theorem implies the following.
\begin{corollary}\label{thm:counting}
  Assumptions as in the Main Theorem, the number of eigenvalues in a sector satisfies
$$
N(\phi_0, \phi_1) = h^{- \alpha (d - 1)/(\alpha - 1)}
(\int_{\phi_0}^{\phi_1} d\mu)(1 + o(1)).
$$
\end{corollary}
This result can be taken as an analogue of the Weyl asymptotic
formula, reviewed below in Section \ref{sec:eigenvalue-accumulation}.
It is also proven at the end of Section \ref{sec:main-theorem}.

The Main Theorem is proven, following \cite{Zelditch-Kuznecov}, via
analysis of the traces of the operators $S^k_h - \Id$.  The fact that
these operators are trace class is shown for example in
\cite{Yafaev:Scattering-Theory:Some-Old}; in Section
\ref{sec:traces-and-compositions} below, we prove a precise
asymptotic formula for the trace which gives its leading order
behavior in $h$ as $h \to 0$.  The trace of $S^k_h - \Id$ is equal to
$h^{-\alpha \gamma} \la \mu_h, p_k(z) \ra$ with $p_k(z) = z^k - 1$,
where $\mu_h$ is the measure defined in the Main Theorem, and our
asymptotic formula for the trace of $S^k_h - \Id$ shows that the Main
Theorem holds for these special values of $f$.  Note that $p_k(z)$
does not, strictly speaking, satisfy the assumptions of the Main
Theorem, as its support contains $1$; in fact, we prove that the
conclusion of the theorem holds on the Banach space of continuous
functions vanishing to first order at $1$ --- see Sections
\ref{sec:main-theorem} and \ref{sec:eigenvalue-distribution}.  To
conclude that the Main Theorem holds we show in Section
\ref{sec:eigenvalue-distribution} that the measures
$\mu_h$ are continuous on this space of continuous functions, which
contain the span of the $p_k$ as a dense subset, and use
an approximation argument to obtain the formula in the Main Theorem.

The trace of $S^k_h - \Id$ is obtained via analysis of the Schwartz
kernel of $S_h$ and its powers.  By \cite{HW2008}, with previous
results for example in \cite{A2005,G1976,RT1989}, the operator $S_h$ is a semiclassical Fourier
Integral Operator whose canonical transformation is the
\textbf{total sojourn relation}, which is a map from incoming rays to
outgoing rays which are asymptotically tangent to the same flow line
of the Hamiltonian system induced by $h^2 \Delta + V - E$.  The
precise relationship between the integral kernel of $S_h$ and the
total sojourn relation is discussed in
Section \ref{sec:the-scattering-matrix}, and in particular we see that
the canonical relation of $S_h$ is a perturbation of the identity operator of
order determined by $\alpha$, the rate of vanishing of $V$ at infinity.

We elaborate the latter remark in the special case that $E = 1$ and
$V$ is central ($V(x) = V(|x|)$) on $\mathbb{R}^2$.  The bicharacteristics of the
Hamiltonian $p = |\xi|^2 + V - 1$ are paths $x(t)$ in $\mathbb{R}^2$
satisfying Newton's equation $\ddot x(t) = - 2\nabla V(x(t))$.  The crucial
object related to the dynamical system in this context is the
`scattering angle' $\Sigma$, the angle by which an incoming ray is
deflected by the potential \cite{RSIII}.  Indeed, in this case the canonical
relation of $S_h$
is the graph of the map
of $T^*\mathbb{S}^1 \lra T^* \mathbb{S}^1$ taking a point $(\omega,
\eta)$ to  $(\omega + \Sigma(\eta), \eta)$, where
$(\omega, \eta)$ corresponds to a straight
ray $x_0(t) = \omega t + \eta$ and $\eta \perp \omega$.  Here, the scattering
angle is given explicitly
by the formula \cite[Eqn.\ 2.6]{DGHH2013}
$$
\Sigma(\eta) = \pi - 2 \int_{r_m}^\infty \frac{\eta}{r^2 \sqrt{1 -
    \eta^2 r^2 - V(r)}} \, dr,
$$
where $r_m$ is the
minimum distance to the origin of the bicharacteristic ray $x(t)$ of the
Hamiltonian $p$ that is asymptotic to $x_0$ for time near minus infinity.  When $V
\sim c /r^\alpha$, it is straightforward to compute that $r_m =
\eta( 1+ O(\eta^{- \alpha}))$ and $\Sigma(\eta) = O(\eta^{-
  \alpha})$.  
  
  \

This is the third paper in a series analyzing the asymptotic distribution
of the phase shifts in the semiclassical limit using geometric
microlocal techniques, the first two works of which consider smooth compactly
supported potentials $V$ \cite{DGHH2013,
  Gell-Redman-Hassell-Zelditch}. It is instructive to compare the Main Theorem with the
  main result of 
\cite{Gell-Redman-Hassell-Zelditch}, which is

\begin{theorem}\label{thm:comp} Let $V$ be a real, smooth, compactly supported potential, and $E \in \RR$
a nontrapping energy for the Schr\"odinger operator $H_h = h^2 \Delta + V - E$. Assume that 
the set of periodic points of powers of the reduced scattering map associated to $H_h$ have measure zero in $T^* \mathbb{S}^{d-1}$. Define the sequence of measures $\nu_h$ on $\mathbb{S}^1$ by 
\begin{equation}
  \label{eq:measure-definition-comp}
  \la \nu_h, f \ra = h^{d-1} \frac1{E^{(d-1)/2} \Vol \mathcal{I}} \sum_{e^{2i \beta_{n, h} } \in \spec(S_h)}
  f(e^{2i \beta_{n, h}}),
\end{equation}
where $\mathcal{I}$ is the  set of $(\omega, \eta) \in T^* \mathbb{S}^{d-1}$ associated to bicharacteristics that meet the support of $V$. 
Then, for every $f \in C^0(\mathbb{S}^1)$ supported away from $1$, we have 
 \begin{equation}
    \label{eq:equicompact}
    \lim_{h \to 0} \la \nu_{h}, f \ra = \frac{1}{2\pi} \int_{0}^{2\pi}
    f(e^{i\phi}) d\phi.
  \end{equation}
In particular, the spectrum of $S_h$ is asymptotically equidistributed on the unit circle $\mathbb{S}^1$, away from the point $1$. 
\end{theorem}

The main differences between the Main Theorem and Theorem~\ref{thm:comp} are
\begin{itemize}
\item The rate of accumulation is different. There are about $h^{-(d-1)}$ eigenvalues in a sector in the case of
compact support
\cite{Gell-Redman-Hassell-Zelditch} as opposed to $h^{-\alpha (d -
  1)/(1 - \alpha)}$ for potentials decaying like $|x|^{-\alpha}$.
  
\item For compactly supported semiclassical
potentials,  the phase shifts \textit{equidistribute} around the unit
circle as $h \to 0$, whereas for polynomial decay \textit{they do
  not}; instead, we get the
homogeneous distributions in the Main Theorem.  

\item The equidistribution result for compactly supported
potentials relies on two dynamical assumptions on the bicharacteristic
flow (the first being non-trapping).  
For polynomial decay, neither of these assumptions are required.
\end{itemize}

These differences arise from the fact that $S_h$ for a compactly supported
potential is semiclassically equal to the identity operator outside a
compact set in phase space, so the difference $S_h^k - \Id$ is the
difference of two semiclassical FIO's of order $0$ with compact
microsupport.  As such their traces grow like $h^{-(d - 1)}$ as $h \to
0$ (see e.g.\ \cite[Appendix]{Gell-Redman-Hassell-Zelditch}), and the
volume of phase space on which $S_h^k - \Id$ is microlocally nontrivial
enters into the leading asymptotics. 

By contrast, in the present setting, compact
subsets of phase space are irrelevant since their contribution to the measure $\mu_h$ is
order $O(h^{-(d-1)} \times h^{\alpha (d-1)/(\alpha-1)} ) = O(h^{(d-1)/(\alpha -1)})$ which is a positive power of $h$, so only the asymptotic behaviour of the dynamics is important. This explains why the nontrapping assumption is not relevant in the present setting, as trapped rays only occur in a compact region of phase space. 

The quicker rate of accumulation of eigenvalues, $\sim h^{-\alpha (d-1)/(\alpha-1)}$, as $h \to 0$ can be understood heuristically by observing that,
for each level of $h$, there is an `effective radius' $r(h) \sim
h^{-1/\alpha}$ outside of which the decaying potential $V$ is $O(h)$,
and therefore semiclassically negligible.
This radius  tends to infinity as a negative power of $h$,
leading to an effective `interacting' volume of phase space that grows
as $h^{-(d-1)/\alpha}$.  Since a unit of phase space volume contributes
roughly $h^{-(d-1)}$ eigenvalues, there are about
$h^{-(d-1)\alpha/(\alpha - 1)} = h^{-\gamma \alpha}$ phase shifts
deflected away from $1$. This
observation also explains why we fail to have equidistribution, as one
might naively guess based on Theorem~\ref{thm:comp}, in the
polynomially decaying case. Namely, there is no firm distinction
between interacting and noninteracting parts of phase space, so
correspondingly, there is no firm division between eigenvalues that
are `essentially $1$' and `essentially different from $1$'. So, unlike
the compactly supported case where the measure divides into a finite,
equidistributed part and an infinite atom at $1$, in the polynomially
decaying case, the point mass at $1$ is `smeared' into an absolutely
continuous measure with infinite mass near $1$. Thus equidistribution
is not possible as it would only account for a finite amount of mass
away from the point $1$.

One technical challenge of this work is that to treat potentials $V$
for which one has only the derivative estimates of the main theorem
requires an extension of the results in \cite{HW2008}.  Indeed, the
structure of the integral kernel of the scattering matrix, the Poisson
operator, and indeed the outgoing and incoming resolvents are treated
in \cite{HW2008} in the case that
$$
V = r^{-2}\sum_{j = 0}^\infty a_j(\omega) r^{-j},
$$
for uniformly bounded $a_j \in C^\infty(\mathbb{S}^{d-1})$. In other
words, in the case that $V$ is a \textit{smooth} function of $\rho =
1/r$ and $\omega$ at $\rho = 0$.  The potentials $V$ under
consideration here are merely \textit{conormal} (see Appendix
\ref{sec:soujourn}) and thus some care is required to show that the
scattering matrix has the FIO structure one would predict by analogy
with the smooth case.  This extension is done in detail in Appendix
\ref{sec:soujourn}.

\

The introduction to \cite{Gell-Redman-Hassell-Zelditch} contains a
literature review on the topic, to which we refer the reader.  In
particular, \cite{SY1985} contains an asymptotic formula for the phase
shifts for \textit{central} potentials of polynomial decay, in this
case with two asymptotic parameters.  Namely, they analyze $\Delta +
gV + k^2$, for large $g$ or $k$, in particular obtaining, when $g =
k^2 = 1/h^2$ a formula for the phase shifts which implies our Main Theorem in the special case of central potentials.  Other
related work includes \cite{BY1982, BY1984, BY1993, Sm1992}.


\section{The scattering matrix, $S_h$}\label{sec:the-scattering-matrix}

We now describe the FIO structure of the scattering matrix for
potentials with polynomial decay.  Note that by setting $\wt{V} = V/E$
and $\wt{h} = h/\sqrt{E}$ we may reduce the general $E$ case to the $E
= 1$ case, and so we assume when convenient that
$$
E = 1.
$$

\subsection{The canonical relation of the scattering matrix}

We now describe the canonical relation of the scattering
matrix.  In fact, following \cite{Gell-Redman-Hassell-Zelditch}, we
define a Legendre submanifold of $T^* \mathbb{S}^{d-1} \times
T^*\mathbb{S}^{d-1} \times \mathbb{R}$ related to the scattering
matrix in a way we describe in detail in Section \ref{sec:Schwartz-kernel}.

Given $\omega' \in \mathbb{S}^{d-1}$ and $\eta' \in \RR^d$ orthogonal to $\omega'$, 
there is a unique
bicharacteristic ray $\gamma_{\omega', \eta'}(t) = (x_{\omega', \eta'}(t) ,
\xi_{\omega', \eta'}(t))$ of the semiclassical
Hamiltonian flow associated with $H_h = h^2 \Delta + V - 1$ satisfying
\begin{equation}\label{eq:incoming-data}
x_{\omega', \eta'}(t) = \omega' t + \eta' + o(1) \mbox{ for } t << 0.
\end{equation}
and $|\xi_{\omega', \eta'}|^2 + V(x_{\omega', \eta'}) \equiv 1$.  (Recall
that a (semiclassical) bicharacteristic ray $\gamma = (x, \xi)$ is a solution to Hamilton's equations
$\dot{x}(t) = 2  \xi$, $\dot{\xi}(t) = - \nabla V(x).$)  By our
non-trapping assumption, as $t \to + \infty$, this ray escapes to infinity, 
taking the form 
\begin{equation}\label{eq:outgoing-data}
x_{\omega', \eta'}(t) = \omega (t - \tau) + \eta + o(1) \mbox{ for } t >> 0.
\end{equation}
Here $\tau = \tau(\omega', \eta')$ is the `time delay'. 
As explained in \cite{Gell-Redman-Hassell-Zelditch}, the pairs $(\omega', \eta')$ 
and $(\omega, \eta)$ can be interpreted as  points in $T^* \mathbb{S}^{d-1}$. 
The map $(\omega', \eta') \mapsto (\omega, \eta)$ is known as the reduced scattering map
$\mathcal{S} = \mathcal{S}_{E =1}$ at energy $E = 1$. (For the
arbitrary energy reduced scattering map see
\cite{Gell-Redman-Hassell-Zelditch} of Appendix \ref{sec:soujourn}.  The ray corresponding to
$(\omega', \eta')$ produces an additional piece
of data, a function $\varphi \colon T^*\mathbb{S}^{d -1} \lra
\mathbb{R}$ defined by\footnote{The function $\varphi$ is closely related to, but not the same as, the time delay function $\tau$.  See \cite[Section 2]{Gell-Redman-Hassell-Zelditch} for further discussion.}
\begin{equation}
  \label{eq:phi-definition}
  \varphi(\omega', \eta') = \int^\infty_{- \infty}  x_{\omega', \eta'}(s) \cdot
  \nabla V(x_{\omega', \eta'}(s))  ds.
\end{equation}
The reduced scattering
map, together with the
map $\varphi$, determine a Legendre submanifold of $T^* \mathbb{S}^{d-1} \times T^* \mathbb{S}^{d-1} \times \RR_\phi$, endowed with the contact form $\eta' \cdot d\omega' + \eta \cdot d\omega - d\varphi$, called the `total sojourn relation' in \cite{HW2008}:
\begin{equation}
  \label{eq:total-sojourn-relation}
  L = \big\{  (\omega', \eta', \omega, -\eta, \phi) \mid (\omega, \eta) = \mathcal{S}(\omega', \eta'), \, \phi = \varphi(\omega', \eta') \big\}.
\end{equation}
That is, the contact form vanishes when restricted to $L$. This
implies immediately that $\mathcal{S}$ is a symplectic map. The
scattering matrix $S_h$ can be described either as a semiclassical
Lagrangian distribution with canonical relation given by the graph of
$\mathcal{S}$, or as a Lagrangian-Legendrian distribution associated
to $L$, as we discuss further in Section~\ref{subsec:oscint}.

We now find it convenient to switch to using local coordinates $y = (y_1, \dots, y_{d-1})$ on the sphere $\mathbb{S}^{d-1}$. We will use $\eta$ for the dual coordinates on the cotangent bundle. This is a slight abuse of notation (compared to the usage of $\eta$ above) but we find it convenient and should not cause confusion. For example, the contact form written in $(y, \eta)$ coordinates is $\eta \cdot dy = \sum_i \eta_i dy_i$.  

Due to the decay of $V$ at spatial infinity, the reduced scattering map $\mathcal{S}$ tends to the identity as $|\eta'|\to \infty$. It will be important in our analysis to understand precisely how this happens. 
In Appendix \ref{sec:soujourn}, we will prove the following
proposition which describes the behavior of the map $\mathcal{S}$ in
the large $\eta$ regime. Before we state the result, we recall
the following standard terminology: we say that a function $\sigma =
\sigma(\sphvar, \dspharv, h)$ is a symbol of order $m$ in $\dspharv$, i.e.\ is in
the space $S^m$, if for each multi-index $k \in \mathbb{N}_0^{d-1}, k'  \in \mathbb{N}_0^{d-1}$,  there is
$C_{k, k'} > 0$ such that
\begin{equation}
  \label{eq:symbol-estimates}
 |  D^{k}_{\sphvar} D^{k'}_\dspharv \sigma | \le C_{k,k'} \la \dspharv \ra^{s - |k'|},
\end{equation}
where $|k| = k_1 + \cdots + k_{d-1}$ and $\la \dspharv \ra =
(|\dspharv|^2 + 1)^{1/2}$ and $C_{k, k'}$ is independent of $h$.

\begin{proposition}\label{thm:sojourn-near-infty}
  Suppose that $\alpha > 1$, and that $V \in \mathbb{C}^\infty$ can be expressed in the form 
  $$
  V = \frac{v_0(\hat x)}{|x|^\alpha} + W,
  $$
where $W$ satisfies \eqref{W-deriv-est}. 
Then the map $(y', \eta') \mapsto (y, \eta, \phi) = (\mathcal{S}(y', \eta'), \varphi(y', \eta'))$ satisfies 
\begin{equation}
  \begin{split}
    \sphvar_i &= \sphvar'_i + a_i(\sphvar', \hat{\dspharv'})
    |\dspharv'|^{- \alpha} + e_i \\
    \dspharv_i &= - \dspharv'_i + b_i(\sphvar', \hat{\dspharv'})
    |\dspharv'|^{1 - \alpha} + \tilde e_i \\
    \varphi &= c(\sphvar', \hat{\dspharv'}) |\dspharv'|^{1 - \alpha} + e',
  \end{split}
\label{conreg}\end{equation}
for each $i = 1, \dots, d-1$ and $|\dspharv'|$ large, where $a_i, b_i$ and $c$ are smooth,  $e_i \in S^{- \alpha - \epsilon}$, and $ \tilde e_i, e' \in S^{1 - \alpha - \epsilon}$ (see \eqref{eq:symbol-estimates}.)
\end{proposition}

\subsection{The Schwartz Kernel of $S_h$}\label{sec:Schwartz-kernel}\label{subsec:oscint}
Let us assume for the next few sections that $E$ is a nontrapping energy, and defer the trapping case to Section~\ref{sec:trapping}. 
As shown in \cite{HW2008}, under this assumption, the scattering matrix is a `Legendrian-Lagrangian distribution'.
This means that for each fixed $h > 0$, $S_h$ is a (homogeneous) FIO; in our case\footnote{The results of \cite{HW2008} apply to asymptotically conic nontrapping manifolds. In general, the absolute scattering matrix is a FIO associated to the canonical relation of geodesic flow (at infinity) at time $\pi$. In the case of $\RR^n$, one obtains the `relative scattering matrix', which is what we consider here, by composing the absolute scattering matrix with the antipodal map and multiplying by $i^{(d-1)/2}$. This reduces the canonical relation to the identity in this special case.}, it is a pseudodifferential operator, in fact equal to the identity up to a pseudodifferential operator of order $1-\alpha$. 
As $h \to 0$, the scattering matrix is, in each bounded region of phase space $T^* \mathbb{S}^{d-1}$, a semiclassical 
Lagrangian distribution, but it is not a semiclassical pseudodifferential operator. Instead, its canonical relation is the graph of the reduced scattering transformation, which is only \emph{asymptotically} equal to the identity; Proposition~\ref{thm:sojourn-near-infty} makes precise how this happens. 
What is new about the result in \cite{HW2008} is that it 
gives the precise oscillatory integral form of $S_h(E)$ in the transitional regime; that is, where the semiclassical frequency $\eta$ tends to infinity, uniformly as $h \to 0$. 
However, more regularity was assumed on the potential $V$ in \cite{HW2008}; the assumption made there translates, in our context, to the potential having a Taylor series at infinity of the form $\sum_{j \geq 2} |x|^{-j} v_j(\hat x)$. We explain how the result extends to the potentials considered here in Appendix~\ref{sec:soujourn}.

\begin{remark}
The reason that the term Legendre distribution is used in \cite{HW2008} is because $S_h(E)$ is associated to the Legendre submanifold \eqref{eq:total-sojourn-relation}, which gives an extra piece of information, namely $\varphi$, in addition to the symplectic map $\mathcal{S}$. This eliminates an ambiguity (up to an additive constant) of the class of phase functions locally parametrizing the associated Lagrangian submanifold $\mathrm{graph }(\mathcal{S})$; see the appendix of \cite{Gell-Redman-Hassell-Zelditch} for more discussion on this point. 
\end{remark}

We now express $S_h(E)$ as an oscillatory integral microlocally near infinity.  

\begin{lemma}\label{lem:S-structure} Suppose $E$ is a nontrapping energy. Then the scattering matrix $S_h$ takes the form 
\begin{equation}\label{eq:Sh-decomposition}
S_h = F_1 + F_2,
\end{equation}
where $F_1$ is a zeroth order FIO with compact microsupport, and $F_2$
has an oscillatory integral representation of the form
\begin{equation}\label{eq:Sh-at-fiber-infinity}
F_2(\sphvar,  \sphvar') = (\frac{1}{2\pi h})^{d - 1} \int e^{i((\sphvar -
  \sphvar') \cdot \dspharv + G(\sphvar', \dspharv))/h} (1 + b(\sphvar,
\sphvar', \dspharv, h)) \, d\dspharv,
\end{equation}
where $G(\sphvar', \dspharv)$ is a symbol of order $1-\alpha$ in $\dspharv$, and 
$b$ is a symbol in $\dspharv$ of order $-\alpha$ (see \eqref{eq:symbol-estimates}). Moreover, 
\begin{equation}\begin{aligned}
G(\sphvar', \dspharv) &= g(\sphvar', \hat \dspharv) |\dspharv|^{1-\alpha} + \tilde g, \quad {\tilde g} \in S^{1-\alpha-\epsilon}
\end{aligned}\label{G} \end{equation}
\end{lemma}
\begin{proof}

We first address the question of finding a phase function parametrizing $L$ near infinity, that is, for large $(\eta, \eta')$. 
By Lemma \ref{thm:sojourn-near-infty}, we may use $(\sphvar', \dspharv)$ as
coordinates on the Legendrian $L$ in the large $|\dspharv|$ region, and we write the remaining coordinates on $L$ in terms of these as
$$
\sphvar = W(\sphvar', \dspharv), \quad \dspharv' = N(\sphvar', \dspharv), \quad \varphi = T(\sphvar', \dspharv),
$$
where $(y', \eta', y, \eta, \varphi) \in L$.
The fact that $L$ is Legendrian implies the following identities
amongst these functions $W, N, T$, arising by expressing the vanishing
of $\eta' \cdot dy' + \eta \cdot dy - d\phi$ in these coordinates: 
\begin{equation}
   N_i d\sphvar'_i  = \sum_{j = 1}^{d -1} (- \dspharv_i \frac{\p W_i}{\p \sphvar'_j} d\sphvar'_j + \frac{\p T_i}{\p \sphvar'_j} d\sphvar'_j ) , \quad 
   \sum_{j = 1}^{d -1 } \dspharv_i \frac{\p  W_i}{\p \dspharv_j}
   d\dspharv_j  - \frac{\p T}{\p \dspharv_i} d\dspharv_i = 0.
\end{equation}
Using these identities, one can check that the Legendrian $L$ is parametrized by the function 
\begin{equation}
\Phi(\sphvar', \sphvar, \dspharv) = (\sphvar - W(\sphvar', \dspharv)) \cdot \dspharv + T(\sphvar', \dspharv).
\label{Phi}\end{equation}
Let us write $W(\sphvar', \dspharv) = \sphvar' + \tilde W(\sphvar', \dspharv)$. 
Then we have, comparing \eqref{G} and \eqref{Phi}, 
\begin{equation}\label{eqn:G-specifically}
G(\sphvar', \dspharv) = - \tilde W(\sphvar', \dspharv) \cdot \dspharv + T(\sphvar', \dspharv).
\end{equation}
It thus suffices to note that by Lemma \ref{thm:sojourn-near-infty}, $\tilde W$ is a classical symbol of order $-\alpha$, and $T$ is a classical symbol of order $1-\alpha$. Compare with \cite[Section 7.2]{HW2008}.

That the scattering matrix has a local oscillatory
integral expression using the phase function $\Phi$ with
symbol by $a = 1 + b$ with $b \in S^{-\alpha}$ is
shown in Appendix \ref{sec:scattering-matrix-deduction}.
\end{proof}

\begin{remark}\label{thm:switch-y-yprime}
The choice to parametrize $L$ using a function $G(y', \eta)$ was  an arbitrary one. We can
just as well use $(y, \eta')$ to furnish coordinates on $L$, and then, 
writing $$y' = W'(\sphvar, \dspharv') = y + \tilde W'(\sphvar, \dspharv') , \quad \dspharv =
N'(\sphvar, \dspharv'), \quad \varphi = T'(\sphvar, \dspharv'),$$ 
the function
$$\Phi'(y', y, \eta') = (y' - y) \cdot \eta' - \tilde W' \cdot \eta' + T'(y, \eta')$$ also
parametrizes $L$.  (To be clear, $\varphi = T'(y, \eta')$ means that $T'(y,
\eta')$ is the value of $\varphi$ on $L$ at the point $(y', \eta', y,
\eta, \varphi)$.) Replacing the dummy variable $\eta'$ by $\overline{\eta} = -\eta'$, and 
setting $$G'(y, \overline{\eta}) = \tilde W'(y,-\overline{\eta}) \cdot \overline{\eta} +
T'(y, -\overline{\eta})$$ gives $\Phi' = (y - y') \cdot \overline{\eta}+ G'(y, \overline{\eta})$, and it
follows as above that $G' = g'(\sphvar', \hat {\overline{\eta}}) |\overline{\eta}|^{1-\alpha} + \tilde
g', \quad {\tilde g'} \in S^{1-\alpha-\epsilon}$. Furthermore, we claim that  
\begin{equation}
g' = g.
\label{gg'}\end{equation}
In fact, since $\tilde W(y', \eta) = - \tilde W'(y, -\overline{\eta})$ and $T(y', \eta) = T'(y, -\overline{\eta})$, 
and writing $\overline{\eta} = \eta + N(y', \eta)$, $N \in S^{1-\alpha}$, we find that
$$
\tilde W(y', \eta) = - \tilde W'\big(y' + \tilde W(y', \eta), -\eta + N(y', \eta) \big),
$$
from which it follows that 
$$
\tilde W(y', \eta) = - \tilde W'(y', -\eta) \text{ modulo } S^{1-2\alpha}.
$$
Similarly, $T'(y', -\eta) - T(y', \eta) \in S^{2(1-\alpha)}$. Thus we see that $G - G' \in S^{2(1-\alpha)}$, from 
which \eqref{gg'} follows immediately.

Thus an oscillatory integral of the form in
\eqref{eq:Sh-at-fiber-infinity} with $G$ as in Lemma~\ref{lem:S-structure} can also be written
\begin{equation}\label{eq:Sh-at-fiber-infinity-y-dependent}
 (\frac{1}{2\pi h})^{d - 1} \int e^{i((\sphvar - \sphvar') \cdot \dspharv +
   G'(\sphvar, \dspharv))/h} (1 + b(\sphvar, \sphvar', \dspharv, h)) \,
 d\dspharv,
\end{equation}
where $G'$ has the same properties as $G$, in fact $G - G' \in S^{2(1-\alpha)}$, and $b$ has the same
symbolic properties as $a$ in \eqref{eq:Sh-at-fiber-infinity}.  
\end{remark}

By adjusting the division between $F_1$ and $F_2$ suitably, we may
assume that $G$, as well as  $\langle \dspharv \rangle^{|\gamma|}
D_\dspharv^\gamma G$, are sufficiently small, which we do without further
comment.  Indeed, choosing a function $\chi = \chi(\dspharv)$ with $\chi
\equiv 1$ for $|\dspharv | < R/2 $ and $\chi \equiv 0$ for $|\dspharv| \ge R$
and writing
\begin{gather*}
  \int e^{i(\sphvar - \sphvar') \cdot \dspharv + G(\sphvar',
    \dspharv)/h} (1 + a) \, d\dspharv = \\
  \qquad \int e^{i(\sphvar - \sphvar') \cdot
    \dspharv + G/h}\chi (1 + a) \, d\dspharv + \int e^{i(\sphvar - \sphvar')
    \cdot \dspharv + G/h} (1 - \chi) (1 + a) \, d\dspharv,
\end{gather*}
 and taking $R$ large enough and including the first term on the right in
$F_1$ produces the desired effect.  We drop the $\chi$ from the
notation in $F_2$ since e.g.\ we could take $a
\equiv  - 1$ for $|\eta| \le R$.

\subsection{Powers of the scattering matrix}
To prove the Main Theorem, we will compute the trace of $S_h(E)^k - \Id$ for all integers $k$. Thus it is important to understand how to represent the powers $S_h^k$ as oscillatory integrals.

First assume that $k \geq 1$. The $k$th power of $S_h$ is $(F_1 +
F_2)^k$, and if we expand this product, every term has compact
microsupport except for $F_2^k$. (See Section
\ref{sec:traces-and-compositions} for a further discussion of the
other terms in the expansion and why their contribution to the eigenvalue
asymptotics is lower order.)  For $k = -1$, recall that $S_h^{-1} =
S_h^*$, and thus the integral kernel of $S_h^{-1}$ is given by the
hermitian conjugates $F^*_1 + F^*_2$ in
\eqref{eq:Sh-decomposition}.  Here $F^*_2(\sphvar,  \sphvar') =
\overline{F}_2(y', y)$.  By Remark \ref{thm:switch-y-yprime}, we may
take the phase function of $F_2$ to be of the form
in~\eqref{eq:Sh-at-fiber-infinity-y-dependent}, specifically
with $G'$ as in \eqref{eq:Sh-at-fiber-infinity-y-dependent}, we have  
\begin{equation}\label{eq:Sh-at-fiber-infinity-conjugate}
F^*_2(\sphvar,  \sphvar') = \overline{F}_2(y', y) = (\frac{1}{2\pi h})^{d - 1} \int e^{i((\sphvar -
  \sphvar') \cdot \dspharv  - G(\sphvar', \dspharv))/h} (1 + b(\sphvar,
\sphvar', -\dspharv, h)) \, d\dspharv,
\end{equation}

We will show that $F^k_2$ has the following oscillatory integral structure.

\begin{lemma}\label{thm:composition} Suppose $k \geq 1$. Then the FIO
  $F_2^k$ has an oscillatory integral representation of the form 
\begin{equation}
(\frac{1}{2\pi h})^{d - 1} \int e^{i((\sphvar - \sphvar') \cdot \dspharv + kG(\sphvar',
  \dspharv) + E_k(\sphvar', \dspharv))/h} (1 + b_k(\sphvar, \sphvar', \dspharv, h))
\, d\dspharv,
\end{equation}
where $E_k$ is a symbol of order $2(1-\alpha)$ in $\dspharv$, and $b_k$ is
a symbol of order $1-\alpha$.  Similarly,   $(F^*_2)^k$ has an oscillatory
integral representation of the form 
\begin{equation}
(\frac{1}{2\pi h})^{d - 1} \int e^{i((\sphvar - \sphvar') \cdot \dspharv - kG(\sphvar',
  \dspharv) + E_k(\sphvar', \dspharv))/h} (1 + b_k(\sphvar, \sphvar', \dspharv, h))
\, d\dspharv,
\end{equation}
\end{lemma}

\begin{proof} See Appendix \ref{sec:composition}. \end{proof}

The key point here is that, up to a term vanishing faster as
$|\eta|\to \infty$, the effect on the phase function $\Phi$ of raising
the scattering matrix to the power $k$ is essentially to replace $G$ by
$kG$.


\section{Proof of the main theorem}\label{sec:main-theorem}
The main idea of the proof, motivated by \cite{Zelditch-Kuznecov, Z1997} and \cite{Gell-Redman-Hassell-Zelditch}, is the following observation: if $\{ \nu_h \} $ is a family of \textit{finite} measures on $\mathbb{S}^1$  parametrized by $h > 0$, and if
each Fourier coefficient of $\nu_h$ converges to that of a certain
finite measure $\nu$ as $h \to 0$, then $\nu_h$ converges to $\nu$ in
the weak-$*$ topology. In our case, however, we cannot apply this
directly, as the $\mu_h$ in \eqref{eq:measure-definition} are infinite measures. Instead we have the following variant. Consider the following weighted sup norm for functions on $\mathbb{S}^1$:
\begin{equation}
  \label{eq:weighted-norm}
  \norm[w]{f} = \sup_{z \in \mathbb{S}^1 \setminus \{ 1 \}} \absv{\frac{f(z)}{z - 1}},
\end{equation}
and the associated Banach space
\begin{equation}
  \label{eq:weighted-space}
  C^0_{w}(\mathbb{S}^1) = \{ f \in C^0(\mathbb{S}^1) : \exists \, g \in C^0(\mathbb{S}^1) \text{ such that } f(z) =  (z-1) g(z) \}.
\end{equation}

Then we show 
\begin{proposition}\label{prop:density}
Suppose that the (infinite) measures $\nu_h$ and $\nu$ act (by integration) as bounded linear functionals on  $C^0_{w}(\mathbb{S}^1)$. Moreover, assume that the norms of $\nu_h$ in the dual space are uniformly bounded in $h$. Then if 
\begin{equation}
\lim_{h \to 0} \int_{\mathbb{S}^1} p(e^{i\theta}) d\nu_h = \int_{\mathbb{S}^1} p(e^{i\theta}) d\nu
\label{p-convergence}\end{equation}
for all polynomials $p \in C^0_{w}(\mathbb{S}^1)$, then $\nu_h \to \nu$ in the weak-$*$ topology on every compact subset of $\mathbb{S}^1 \setminus \{ 1 \}$. 
\end{proposition}

\begin{proof} 
The proof of this proposition consists of two elementary steps. We first observe, as in \cite[Proof of Lemma 5.3] {Gell-Redman-Hassell-Zelditch}, that polynomials in $C^0_{w}(\mathbb{S}^1)$ are dense in $C^0_{w}(\mathbb{S}^1)$. The proof is so simple that we repeat it here: given $f \in C^0_{w}(\mathbb{S}^1)$, by definition $f = (z-1) g(z)$ for some $g \in C^0(\mathbb{S}^1)$. We approximate $g$ in $C^0(\mathbb{S}^1)$ by a sequence of polynomials $p_j$. Then $(z-1) p_j$ lie in $C^0_{w}(\mathbb{S}^1)$ and approximate $f$ in the $C^0_{w}(\mathbb{S}^1)$ norm.

Then, given $\epsilon > 0$, and a continuous function $f$ on the circle, supported away from $1$, we need to show that 
$$
\Big| \int f d\nu - \int f d\nu_h \Big| < C\epsilon,
$$
provided $h$ is sufficiently small. We choose a polynomial $p$ in $C^0_{w}(\mathbb{S}^1)$ such that 
$\| p - f \|_{[w]} < \epsilon$. Then we estimate 
\begin{equation}\begin{gathered}
\Big| \int f d\nu - \int f d\nu_h \Big| \\
\leq \Big| \int f d\nu - \int p d\nu \Big| + \Big| \int p d\nu - \int p d\nu_h \Big| + \Big| \int p d\nu_h - \int f d\nu_h \Big|.
\end{gathered}\end{equation}
The first term is bounded by $C_1 \| p - f \|_{[w]}$ where $C_1 = \| \nu \|_{(C^0_{w})^*}$ is the dual norm of $\nu$. The second term is bounded by $\epsilon$ provided that $h$ is sufficiently small, using \eqref{p-convergence}. The third term is bounded by $C_2 \| p - f \|_{[w]}$ where $C_2$ is a uniform bound on  $\| \nu_h \|_{(C^0_{w})^*}$. Taking $C = C_1 + C_2 + 1$, this completes the proof. 
\end{proof}

In the case of interest, 
$\nu_h$ will be the measure $\mu_h$ defined in \eqref{eq:measure-definition} and $\nu$ will
be the pushforward of a homogeneous measure. In view of
Proposition~\ref{prop:density} we need the following result.

\begin{proposition}\label{thm:trace-class}
  There exists $c > 0$ such that for $h$ sufficiently
  small
  \begin{equation}
    \label{eq:1}
   | \la \mu_h , f \ra | \le c \norm[w]{f}.
  \end{equation}

\end{proposition}

We also need 

\begin{proposition}\label{cor:traceclass}
For every $k \in \ZZ$, $S_h^k - \Id$ is trace class.
\end{proposition}

Propositions~\ref{thm:trace-class} and \ref{cor:traceclass} will be proved in Section \ref{sec:eigenvalue-distribution}. 

It is easy to see that the polynomials $z^k - 1$, for $k \neq 0 \in \ZZ$, form a basis of the polynomials that vanish at $1$. 
We thus need to show that, for each $k$,  the quantity 
\begin{equation}
\int (e^{ik\theta} - 1) d\mu_k = \Trace (S_h^k - \Id)
\label{k-Four-coeff-h}\end{equation} 
converges as $h \to 0$, and to find a limit measure $\mu$ such that the limit of \eqref{k-Four-coeff-h} is equal to 
\begin{equation}
\int (e^{ik\theta} - 1) d\mu.
\label{k-Four-coeff}\end{equation}

It turns out that the limit of the quantity \eqref{k-Four-coeff-h} is given by a power $c_{\pm} |k|^\gamma$, where the coefficient $c_\pm$ depends only on the sign of $k$. Such a homogenous `Fourier series' comes from a `homogeneous measure'. We now describe precisely what this entails. 

\begin{definition}\label{def:hommeas} Let $\beta < -1$. We say that a measure $\mu$ on $\mathbb{S}^1 \setminus \{ 1 \}$ is the pushforward of a homogeneous measure of degree $\beta$ if there is a measure $\nu$ on $\RR$ of the form 
\begin{equation}
\nu = \begin{cases}
c_1 \theta^{\beta} d\theta, \quad  \, \theta > 0 \\
c_2 |\theta|^{\beta} d\theta, \quad \theta < 0,
\end{cases}
\label{hom-measures}\end{equation}
such that $\mu$ is the pushforward of $\nu$ under the quotient map $F : \RR \mapsto \mathbb{S}^1 = \RR / 2\pi \ZZ$:
$$
\mu = F_*(\nu) \text{ on } \mathbb{S}^1 \setminus \{ 1 \}.
$$
\end{definition}

To state the following lemma, we will need to define the constant 
\begin{equation}\label{eq:Gamma}
\Gamma = \int_0^\infty (e^{i \theta} - 1) \theta^{- \gamma - 1} d\theta.
\end{equation}

for $0 < \gamma < 1$.  Note that $\Gamma$ is finite for $\gamma$ in
this range.
\begin{lemma}\label{lem:hom-measures}  Suppose that $\mu$ is a measure on $\mathbb{S}^1$ that lies in the dual space of $C^0_w(\mathbb{S}^1)$, and is such that, for some
$\gamma > 0$,  
\begin{equation}
\int_{\mathbb{S}^1} \big( e^{ik\theta} - 1 \big) d\mu = 
\begin{cases} (\Gamma c_1 + \overline{\Gamma} c_2) k^{\gamma}, \quad \ \ k = 1, 2, 3, \dots \\
(\overline{\Gamma} c_1 + \Gamma c_2) |k|^{\gamma}, \quad k = -1, -2, -3, \dots 
\end{cases}.
\label{k-hom}\end{equation}
Then on $\mathbb{S}^1 \setminus \{ 1 \}$,  $\mu$ is the pushforward of a
homogeneous measure of degree $\beta = -1-\gamma$. Indeed, it is given
by $\nu$ in \eqref{hom-measures} (with the same constants $c_1$ and $c_2$).
\end{lemma}

\begin{remark} In this lemma, we are only making a statement about
  $\mu$ away from the point $1$. Notice that there could be an atom at $1$ about which we can say nothing, as this would not affect the integrals in \eqref{k-hom}.  \end{remark}

\begin{proof}
We first note that the integrals in \eqref{k-hom} determine $\mu$ uniquely as an element of the dual space of $C^0_w(\mathbb{S}^1)$, hence uniquely as a measure away from the point $1$. This is an immediate consequence of the density of polynomials in $C^0_w(\mathbb{S}^1)$. In view of this, it suffices to show that the measures in Definition~\ref{def:hommeas}, homogeneous of degree $\beta = -1-\gamma$, have `Fourier coefficients' of the form \eqref{k-hom}.

With $\beta = - \gamma - 1$, let $\nu = H(\theta) \theta^{\beta} d\theta$ where $H$ is the
Heaviside function, and let $\mu$ be the pushforward of $\nu$. We compute, for $k > 0$, 
\begin{equation}\begin{gathered}
\int_{\mathbb{S}^1} \big( e^{ik\theta} - 1 \big) d\mu = \int_{\RR} \big( e^{ik\theta} - 1 \big) d\nu  \\
= \int_0^\infty \big( e^{ik\theta} - 1 \big) \theta^{\beta} d\theta \\
= k^\gamma \int_0^\infty \big( e^{i\theta} - 1 \big) 
\theta^{\beta} d\theta 
= k^\gamma \Gamma.
\end{gathered}\end{equation}
Similarly, for $k < 0$ we find 
$$
\int_{\mathbb{S}^1} \big( e^{ik\theta} - 1 \big) d\mu =  |k|^\gamma \int_0^\infty \big( e^{-i\theta} - 1 \big) 
\theta^{\beta} d\theta = |k|^\gamma \overline{\Gamma}.
$$
Similarly, if $\nu = (1 - H(\theta)) |\theta|^\beta$, then for $k > 0$
\begin{equation*}\begin{gathered}
\int_{\mathbb{S}^1} \big( e^{ik\theta} - 1 \big) d\mu = \int_{-\infty}^0 (e^{ik\theta} - 1) |\theta|^\beta d\theta =  \int_0^\infty \big( e^{-ik\theta} - 1 \big) \theta^\beta d\theta  
= k^\gamma \overline{\Gamma},
\end{gathered}\end{equation*}
and $\int_{\mathbb{S}^1} \big( e^{ik\theta} - 1 \big) d\mu = |k|^\gamma \Gamma$ for $k < 0$.
\end{proof}

We can now prove the Main Theorem.
\begin{proof}[Proof of the Main Theorem]
  Proposition~\ref{thm:trace-class} shows that the measures $\mu_h$
  are uniformly bounded in the dual space of $C^0_{w}(\mathbb{S}^1)$
  (since $\alpha > d$). This means that we can
  apply Proposition~\ref{prop:density}, showing that $\mu_h$
  converges to a measure $\mu$ on $C^0_{w}$ provided that the
  convergence on polynomials vanishing at $1$ in  \eqref{p-convergence}
  holds. Since the
  polynomials $e^{ik\theta} - 1$ for $k \neq 0 \in \ZZ$ are a basis
  for polynomials vanishing at $1$, it suffices to check
  \eqref{p-convergence} for these polynomials. So this requires
  computing the limit, as $h \to 0$, of $\Trace (S_h^k - \Id)$. This
  we shall do in the Section \ref{sec:traces-and-compositions}, and the result is \eqref{trace2}
  and \eqref{trace3}. That is, the integrals are given by
  $(\sqrt{E}/2 \pi)^{d-1}ck^{\gamma}$ for $k >0$ and $(\sqrt{E}/2 \pi)^{d-1}\overline{c}k^{\gamma}$ for $k < 0$,
  where $c$ is the constant in \eqref{eq:constants-3}, in particular
  $c = a_1 \Gamma + a_2 \overline{\Gamma}$ for $a_1, a_2$ in
  \eqref{eq:constants-3}. Thus by Lemma~\ref{lem:hom-measures}, the
  homogeneous measure $\mu$ given by the pushforward of $a_1(\sqrt{E}/2 \pi)^{d-1} H \theta^\beta + a_2 (\sqrt{E}/2 \pi)^{d-1} (1 - H)
  |\theta|^\beta$ which pairs with $e^{ik \theta} - 1$ to give the
  above values. \end{proof}

We now prove Corollary \ref{thm:counting}.
\begin{proof}[Proof of Corollary \ref{thm:counting}]
  Given $0 < \phi_0 < \phi_1 < 2\pi$, and let $1_{[\phi_0, \phi_1]} \colon
  \mathbb{S}^1 \lra \mathbb{R}$ be the indicator function of the
  corresponding sector of the circle, $1_{[\phi_0, \phi_1]}(e^{i
    \theta}) = 1$ if $\phi_0 \le \theta \le \phi_1$ modulo $2\pi$ and
  is zero otherwise.  Then for $N(\phi_0, \phi_1)$ as defined in the corollary
$$
N(\phi_0, \phi_1) = \Tr(1_{[\phi_0, \phi_1]}(S_h))
$$
Let $f$ and $g$ are continuous, non-negative
  functions, on the circle supported on $\mathbb{S}^1 \setminus 1$
  such that $ f \le 1_{[\phi_0, \phi_1]} \le g$.  Then 
$$
\Tr f(S_h) \le \Tr(1_{[\phi_0, \phi_1]}(S_h)) \le \Tr g(S_h),
$$
and all three quantities are finite.  Thus
$$
\int f \mu + o(1) = \la \mu_h, f \ra \le h^{\alpha \gamma}  N(\phi_0,
\phi_1) \le  \la \mu_h, g \ra = \int g \mu + o(1),
$$
where $o(1)$ denotes a quantity which goes to $0$ as $h \to 0$.
But for any $\epsilon > 0$ we can choose $f$ and $g$ can be chosen so
that $\int f \mu = \int 1_{[\phi_0, \phi_1]} \mu - \epsilon, \int g \mu
= \int 1_{[\phi_0, \phi_1]} \mu + \epsilon$, and thus for any $\epsilon$
we have 
$$
\int 1_{[\phi_0, \phi_1]} \mu - \epsilon + o(1)  \le h^{\alpha \gamma}
N(\phi_0, \phi_1)  \le \int 1_{[\phi_0, \phi_1]} + \epsilon+ o(1),
$$
so $\lim h^{\alpha \gamma}N(\phi_0, \phi_1)  = \int 1_{[\phi_0,
  \phi_1]} \mu$, proving the Corollary.
\end{proof}


\section{Traces and compositions}\label{sec:traces-and-compositions}

We now compute the traces
\begin{equation}
\lim_{h \to 0} h^{\alpha \gamma} \tr (S_h^k - \Id),
\label{trace}\end{equation}
again with $\gamma = (d - 1)/(\alpha - 1)$, assuming still that
$\alpha > d$.  
Indeed, we
will prove
\begin{lemma}\label{thm:trace-formula}
There is a constant $c$ such that for $k \in
  \mathbb{Z}$, 
$$
\lim_{h \to 0} \la \mu_h, z^k - 1 \ra = \left\{
  \begin{array}{cc}
     c (2\pi)^{-(d - 1)}  k^\gamma & \mbox{ if } k \ge 0 \\
     \overline{c}  (2\pi)^{-(d - 1)} |k|^\gamma &\mbox{ if } k < 0
  \end{array} \right.
$$
Indeed, $c = a_1 \Gamma + a_2 \overline{\Gamma}$ where $a_1$ and $a_2$
are defined in \eqref{eq:constants-3}
\end{lemma}
Note that any semiclassical, zeroth order FIO with compact
microsupport has trace bounded by $C h^{-(d-1)}$ (see e.g.\
\cite[Appendix]{Gell-Redman-Hassell-Zelditch}.) Therefore, in the
limit above we may replace $S_h^k$ by $F_2^k$.  Indeed, in
$$
S_h^k - \Id = F_2^k - \Id + \sum_{j = 0}^{k - 1} {j \choose k} F_1^{k - j} F_2^j,
$$
all the terms in the sum on the right have compact microsupport, are
thus trace class with trace bounded by $h^{-(d -1)}$, and as we will
see, the trace of $F_2^k - \Id$ increases at the rate $h^{-\alpha
  \gamma}$.  Since
$$
\alpha \gamma = \alpha (d - 1)/(\alpha - 1) > d - 1,
$$
the trace of $F_2^k - \Id$ contributes the leading order part of
$\Tr(S_h^k - \Id)$.
 For the same reason, we
can replace the identity operator by another FIO that differs from it
by an operator with compact microsupport. So we can restrict to the
microlocal region $|\dspharv| \geq R$ for arbitrary $R$, and using $h^{\alpha \gamma}
h^{-(d - 1)} = h^\gamma$ and Lemma \ref{thm:composition}, we can write
the terms Schwartz kernel of $S_h^k - \Id$ which contribute to the
trace to leading order in the form
\begin{equation}\label{eq:what-you-are-taking-trace-of}
 \int\limits_{|\dspharv| \geq R} e^{\frac{i}{h} \big( (\sphvar - \sphvar') \cdot \dspharv  + kG(\sphvar, \dspharv) \big)} \Big( 1 + b_k(\sphvar, \sphvar', \dspharv, h) \Big)   \, d\dspharv 
- \int\limits_{|\dspharv| \geq R} e^{\frac{i}{h}  (\sphvar - \sphvar')
  \cdot \dspharv  } \, d\dspharv,
\end{equation}

We then compute \eqref{trace} by setting $y=y'$ and integrating over
$y$.  To be more precise, the Schwartz kernel corresponding to the
oscillatory integral expression in
\eqref{eq:what-you-are-taking-trace-of} is actually a
\textit{half-density} on $\mathbb{S}^{d-1} \times \mathbb{S}^{d-1}$
acting on half-densities on $\mathbb{S}^{d-1}$, meaning that if, for
the moment, $H(y, y')$ denotes the distribution in
\eqref{eq:what-you-are-taking-trace-of}, then $H(y, y') |dy
dy'|^{1/2}$ is the Schwartz kernel -- more precisely the Schwartz
kernel is a finite sum of these -- and it acts on half densities
$\phi(y)|dy|^{1/2}$ by
\begin{equation}\label{eq:half-densities}
\phi(y)|dy|^{1/2} \mapsto (\int H(y, y') \phi(y') |dy'|) |dy|^{1/2}.
\end{equation}
 It is standard that the trace of this operator is $\int H(y, y)
 |dy|$, and thus we are tasked with computing
\begin{equation}
 \lim_{h \to 0} \frac{h^{\gamma}}{ (2\pi)^{d-1}}  \Bigg(  \int\limits_{|\dspharv| \geq R} e^{ikG(\sphvar, \dspharv)/h}  a_k(\sphvar, \sphvar, \dspharv, h)  \, dy   \, d\dspharv 
+ \int\limits_{|\dspharv| \geq R} \Big( e^{ikG(\sphvar, \dspharv)/h} - 1 \Big) \, dy \, d\dspharv \Bigg) 
\end{equation}
Since $a_k= O(|\dspharv|^{1-\alpha})$ and $\alpha > d$, the first
integral is absolutely convergent. Due to the positive power of $h$
out the front, this term is zero in the limit $h \to 0$.

So consider the second term. We write
\begin{equation}\label{eq:G-decomposition}
G(\sphvar, \dspharv) = |\dspharv|^{1-\alpha} g(\sphvar, \hat \dspharv) + |\dspharv|^{1-\alpha-\epsilon}  \tilde g(\sphvar, \dspharv),
\end{equation}
where $\wt{g}$ is a symbol of order $0$. We change variable to $ \dspharv'
= \dspharv (h/k)^{1/(\alpha - 1)}$ to obtain
\begin{equation*}
  \begin{split}
&\frac{h^{\gamma}}{(2\pi)^{d-1}} \int\limits_{|\dspharv| \geq R} \Big( e^{ikG(\sphvar, \dspharv)/h} - 1 \Big) \, dy \, 
d\dspharv \\
& \qquad =     \frac{k^{\gamma}}{(2\pi)^{d-1}} \int \Big( e^{i \big( g(\sphvar,
      \hat\dspharv')|\dspharv'|^{1-\alpha} + h^{\epsilon/(\alpha - 1)} k^{- \epsilon/(\alpha - 1)} \tilde
      g(\sphvar, \dspharv' (k/h)^{1/(\alpha - 1)}) \big)} - 1 \Big) \, dy \,  d\dspharv'.
  \end{split}
\end{equation*}
The integrand in this integral is dominated for small $h$ by
$C|\dspharv'|^{1-\alpha}$ for $C > 0$ independent of $h$, which is integrable as $\alpha > d$. Thus, by the dominated convergence theorem, we can take the pointwise limit inside the integral, and obtain 
\begin{equation}
\lim_{h \to 0} \la \mu_h, z^k - 1 \ra = \lim_{h \to 0} h^{(d-1)\alpha/(\alpha - 1)} \tr (S_h^k - \Id) =
c (1/(2\pi))^{d-1} k^{\gamma}, \quad 
\label{trace2}\end{equation}
where
\begin{equation}
c  = \int \Big( e^{i  g(\sphvar, \hat\dspharv')|\dspharv'|^{1-\alpha} } - 1
\Big) \, dy \, d\dspharv', \quad k \geq 1.\label{eq:constant}
\end{equation}

The fact that $S_h$ is unitary immediately implies 
\begin{equation}
\lim_{h \to 0} h^{(d-1)\alpha/(\alpha - 1)} \tr (S_h^k - \Id) =
\overline{c} (1/(2\pi))^{d-1} |k|^{\gamma} \quad k \leq -1.  
\label{trace3}\end{equation}

It remains only to evaluate $c$. Write
$$
g(\sphvar, \hat\dspharv') = g_+(\sphvar, \hat\dspharv') - g_-(\sphvar,
\hat\dspharv'),
$$
where $g_+ = \max\{ g, 0\}, g_- = \max\{- g, 0\}$.  Then
\begin{equation*}
  \begin{split}
    \int \Big( e^{i g(\sphvar, \hat\dspharv')|\dspharv'|^{1-\alpha} }
    - 1 \Big) \, dy \, d\dspharv' &= \int \Big( e^{i g_+(\sphvar,
      \hat\dspharv')|\dspharv'|^{1-\alpha} } - 1 \Big) \, dy \,
    d\dspharv' \\
    &\qquad + \int \Big( e^{ - i g_-(\sphvar,
      \hat\dspharv')|\dspharv'|^{1-\alpha} } - 1 \Big) \, dy \,
    d\dspharv'.
  \end{split}
\end{equation*}
Considering the first term on the right hand side, we write $\dspharv' = r
\hat\dspharv'$ with $\hat\dspharv'  = \dspharv'/ |\dspharv'|$ and for fixed $\sphvar, \hat\dspharv'$ with $g_+(\sphvar,
\hat\dspharv') \neq 0$, compute
\begin{equation}
  \label{eq:constants-1}
  \begin{split}
   & \int \Big( e^{i  g_+(\sphvar, \hat\dspharv')|\dspharv'|^{1-\alpha} } - 1
\Big) d\dspharv' \\
&\qquad = \int \Big( e^{i  r^{1 - \alpha} } - 1
\Big) g_+^\gamma r^{d - 2} dr d\hat\dspharv' d\sphvar \\
&\qquad =  \int g_+^\gamma \lp \int ( e^{i  r^{1 - \alpha} } - 1)
 r^{d - 2} dr \rp d\hat\dspharv' d\sphvar \\
&\qquad =  \int g_+^\gamma \lp \int ( e^{i \rho } - 1)
 \frac{1}{1 - \alpha}  \rho^{- \gamma - 1} d\rho \rp d\hat\dspharv'
 d\sphvar \\
&\qquad =  \frac{\Gamma}{1 - \alpha} \int g_+^\gamma d\hat\dspharv' d\sphvar 
  \end{split}
\end{equation}
where in the first step we set $r = g_+^{1/(1 - \alpha)} \wt{r}$,
in the third step we set $r^{1 - \alpha} = \rho$, and where $\Gamma$
is the constant defined in \eqref{eq:Gamma}.  A similar computation
shows that 
\begin{equation}
  \label{eq:constants-2}
  \int \Big( e^{ - i  g_-(\sphvar, \hat\dspharv')|\dspharv'|^{1-\alpha} } - 1
  \Big) d\dspharv'  =  \frac{\overline{\Gamma}}{1 - \alpha} \int g_-^\gamma d\hat\dspharv' d\sphvar.
\end{equation}
and thus it follows that
\begin{equation}
c = a_1 \Gamma + a_2 \overline{\Gamma} , \mbox{ where } a_1 =
\frac{1}{1 - \alpha} \int g_+^\gamma d\hat\dspharv' d\sphvar, \mbox{ and }  a_2 =
\frac{1}{1 - \alpha} \int g_-^\gamma d\hat\dspharv'
d\sphvar \label{eq:constants-3}
\end{equation}


\section{The asymptotic distribution of phase shifts}\label{sec:eigenvalue-distribution}

The aim of this section is to prove Propositions \ref{thm:trace-class}
and \ref{cor:traceclass}, to which end we must first obtain an
estimate for the rate at which eigenvalues of $S_h$ accumulate at $1$

\subsection{Eigenvalue accumulation at $1$}\label{sec:eigenvalue-accumulation}

\begin{proposition}\label{thm:rough-eigenvalue-asymptotics}
  Let $\{ e^{2 i\beta_{h, n}} \}_{n = 1}^\infty$ be the eigenvalues of the scattering
  matrix $S_h$.  There exists a
  constant $c > 0$ such that for each $\epsilon > 0$ sufficiently
  small,
  \begin{equation}
    \label{eq:away-from-1}
    \# \{ n : | e^{2 i\beta_{h, n}} - 1 | > \epsilon \} \le c
    \epsilon^{-\gamma} h^{-\alpha \gamma},
  \end{equation}
where as above $\gamma = (d - 1)/(\alpha - 1)$.
\end{proposition}

Before we prove the proposition, we remind the reader of the following
fact; let $K_h$ be a semiclassical pseudodifferential operator of
semiclassical order $0$ on a compact manifold $N$ without boundary,
with compact microsupport.  Then there
exists a $c > 0$ such that off a subspace $W_h \subset L^2$ with $\dim
W_h \le c h^{- \dim N}$, we have
\begin{equation}
  \label{eq:pseudo-bound}
  \| K_h \|_{W_h^\perp \to L^2} = O(h^\infty).
\end{equation}
Indeed, this can be shown using properties of the semiclassical
Laplacian $h^2 \Delta_N$ corresponding to a Riemannian metric on $N$
as follows.
If one considers a smooth, compactly supported function $\chi
\colon T^* N \lra \mathbb{R}$, with $\WF_{h}(K_h) \subset \supp \chi$, then
writing
$$
K_h = K_h \chi(h^2 \Delta) + K_h (\Id - \chi(h^2 \Delta)),
$$
the term on the left hand side satisfies $\| K_h \chi(h^2 \Delta)
\| \le \| K_h  \| \|\chi(h^2 \Delta)\|$ where the norms are operator
norms as maps on $L^2$.  (see \cite{Zworski-book} for a definiton of
the semiclassical wavefront set $WF_h$ of a semiclassical FIO.) 
By the Weyl
asymptotic formula for semiclassical pseudodifferential operators of
order $2$ \cite[Section 6.4]{Zworski-book}, which states that, if $\{\lambda^2_j \}_{j
  = 1}^\infty$  are
the eigenvalues of $\Delta$ 
\begin{equation}\label{eq:Weyl-asymptotics}
\# \{ j : \lambda_j^2 < \lambda^2 \}  = c_N \lambda^{\dim N} +
O(\lambda^{dim N - 1}),
\end{equation}
as $\lambda \to \infty$, we see that
$\chi(h^2 \Delta)$ is identically zero off a
set of dimension no bigger than $c \Vol(\supp (\chi)) h^{-\dim N}$.
For the term on the left, $\Id - \chi(h^2 \Delta)$ is a semiclassical
pseudo of order $0$, with microsupport in $(\supp \chi)^{c}$, and thus
$$
\WF_{h} (K_h (\Id - \chi(h^2 \Delta))) \subset \WF_{h}(K_h) \cap
(\supp \chi)^c = \varnothing,
$$
and in particular $\| K_h(\Id - \chi(h^2 \Delta)) \| = O(h^\infty)$.

In the proof we will use a semiclassical version of the
Calderon-Vaillancourt theorem, which, in the non-semiclassical setting
\cite[Theroem
2.73]{Folland-Harm-Anal-in-Phase} states that a
pseudodifferential 
operator $Q$ on $\mathbb{R}^n$,
$$
Q = \int e^{(z - z') \cdot \zeta} a(z, z', \zeta) d\zeta,
$$
satisfying
$$
\| a \|_{2n + 1} := \sup_{|\alpha| + |\beta| \le 2n + 1}  \|
\p^\alpha_{z, z'} \p^\beta_\zeta a \|_{L^\infty} < \infty
$$ 
is bounded on $L^2$ with
$  \| Q \|_{L^2 \to L^2} \le c_n \| a \|_{2n + 1}$,
where $c_n$ is a constant depending only on the dimension $n$.
Setting $\zeta = \wt{\zeta} /h$ shows that for a semiclassical
pseudodifferential operator of order $0$,
$$
Q_h = h^{-n}\int e^{(z - z') \cdot \zeta / h} b(z, z', \zeta, h) d\zeta,
$$
if
\begin{equation}\label{eq:CV-norm}
\| b \|_{2n + 1, h} := \sup_{|\alpha| + |\beta| \le 2n + 1}  \|
\p^\alpha_{z, z'} (h\p_\zeta)^\beta b \|_{L^\infty} < \infty
\end{equation}
is bounded on $L^2$ with norm bounded by $c_n \| b \|_{2n + 1, h}$.
Indeed, this is just Courant-Vaillancourt with symbol depending on a
smooth parameter applied to a semiclassical
symbol $b(z, z', \zeta, h)$ 

\begin{proof}[Proof of Proposition \ref{thm:rough-eigenvalue-asymptotics}]
  Consider the operator $S_h = F_1 + F_2$ decomposed as in
  \eqref{eq:Sh-decomposition}, where $F_1$ has compact microsupport
  and $F_2$ consists of a finite sum of terms of the form
  \eqref{eq:Sh-at-fiber-infinity}.  We begin by taking a cutoff
  function $\chi_R \colon T^*\mathbb{S}^{d -1 } \lra \mathbb{R}$
  with $\chi_R(\dspharv) \equiv 1 $ for $|\dspharv| \le R/2$ and $\supp \chi
  \subset \{ |\dspharv | \le R \}$.  Taking $R$ sufficiently large we have
$$
S_h - \Id  = (S_h - \Id) \Op_h(\chi_R) + (S_h - \Id) (\Id - \Op_h(\chi_R)),
$$
where $\Op_h(q)$ applied to a symbol $q \in S^m(\mathbb{S}^{d -1})$
denotes the right quantization of $q$ to a semiclassical
pseudodifferential operator of order $m$, again see
\cite{Zworski-book}.  Here,
the operator $(S_h - \Id) \Op_h(\chi) \in \Pscl^0$, and thus the statments
preceeding the proof apply, and since $\Op_h(\chi_R)$ is a semiclassical
pseudodifferential operator of order $0$ with compact microsupport
\cite[Appendix]{Gell-Redman-Hassell-Zelditch},  again using $d -1 <
\alpha \gamma$, we see by the discussion
proceeding the proof that the behavior of $(S_h - \Id) \Op_h(\chi)$
has no bearing on \eqref{eq:away-from-1}.
  On the other hand, we can take $R$ large
enough so that $\WF_{h} (F_1 \Op_h(\chi_R)) = \varnothing$ and thus
the Schwartz kernel of the operator
$$
A_h = (S_h
- \Id) (\Id - \Op_h(\chi_R)).
$$ 
is a sum of terms of the form \eqref{eq:Sh-at-fiber-infinity} plus
terms of order $O(h^\infty)$, and we thus focus our attention on terms
as in \eqref{eq:Sh-at-fiber-infinity}.

By Remark \ref{thm:switch-y-yprime}, we may take $F_2$ to be as in
\eqref{eq:Sh-at-fiber-infinity-y-dependent} with phase function $\Phi = (\sphvar - \sphvar') \cdot \dspharv
+ G(\sphvar, \dspharv)$ which appear in $F_2$ have the asymptotics $G =
a(\sphvar, \wt{\dspharv})|\dspharv|^{1 - \alpha} + O(|\dspharv|^{-\alpha})$.  For constant $\delta
> 0$, we consider two
asymptotic regimes
\begin{equation}
  \label{eq:regimes}
  \begin{split}
    \mbox{regime I:}\quad |\dspharv| h^{1/(1 - \alpha)} \ge \delta &\quad \mbox{ here }
    \exp{iG(\sphvar, \dspharv)/ h} \mbox{ is \textit{not} oscillatory} \\
    \mbox{regime II:} \quad |\dspharv| h^{1/(1 - \alpha)} < \delta &\quad \mbox{ here }
   \exp{ i G(\sphvar, \dspharv)/ h} \mbox{ is  oscillatory.} 
  \end{split}
\end{equation}
As in the discussion proceeding the proof, we use functions of the
semiclassical Laplacian 
$$
P_h := h^2 \Delta_{\mathbb{S}^{d -1}}.
$$
Let $\chi\colon \mathbb{R}^+ \lra \mathbb{R}$ be a bump function with
$\chi(r) = 1$ for $r \le 1/2$, $\chi \ge 0$, and $\supp \chi \subset
[0, 1]$, and write 
\begin{equation}
  \label{eq:Ah-bustup}
  A_h = A_{h, 1} + A_{h, 2},
\end{equation}
where
$$ A_{h, 1} := A_h \chi((\epsilon
h)^{1/(\alpha - 1)} P_h) , \qquad 
A_{h, 2} := + A_h (\Id - \chi((\epsilon
h)^{1/(\alpha - 1)} P_h)).
$$
 We analyze these two operators
separately.

For $A_{h, 1}$, we begin by
pointing out that there exist subspaces $V_h\subset L^2$ with $\dim V
< \epsilon^{- \gamma} h^{-\alpha \gamma}$ such that  
$(\chi((\epsilon h)^{2/(\alpha -  1)} P_h) \rvert_{V_h^\perp} \equiv
0$.  Indeed, by the Weyl asymptotic formula
\eqref{eq:Weyl-asymptotics}, the direct sum of the 
eigenspaces of $(\epsilon h)^{2/(\alpha -  1)} P_h$ with eigenvalue
less than $1$, which we take to be $V_h$, satisfies
\begin{equation}\label{eq:count-of-big-phase-shifts}
\dim V_h \le c ((\epsilon h)^{- 1/(\alpha - 1)} h^{-1})^{-(d - 1)} = c
  \epsilon^{-\gamma} h^{-\alpha \gamma}.
\end{equation}
Thus $A_{h, 1}$ is also identically
zero off $V_h$.

Now consider $A_{h, 2} := A_h (\Id - \chi((\epsilon h)^{1/(\alpha - 1)}
P_h))$.  We claim that there exists a constant $C > 0$ independent of
$h$ and $\epsilon$ such that
\begin{equation}\label{eq:epsilon-norm-bound}
\| A_{h, 2} \| \le C \epsilon.
\end{equation}
The Schwartz kernel of the operator $\chi((\epsilon
h)^{1/(\alpha - 1)} P_h)$ is given by finite sums of terms
\begin{equation}\label{eq:microlocal-function-of-Lapl}
(2 \pi h)^{-(d -1)} \int e^{i (\sphvar - \sphvar')\cdot \dspharv/h }\wt{\chi}(\sphvar, \sphvar', \dspharv (\epsilon h)^{1/(\alpha - 1)}, h)   d\dspharv,
\end{equation}
where $\wt{\chi}(\sphvar, \sphvar', \wt{\dspharv}, h)$ is a semiclassical symbol of
order zero with $\wt{\chi} \rvert_{|\wt{\dspharv}| \ge 1} = O(h^\infty)$.
The operator $A_{h, 2}$ is given by terms of the form
\begin{equation}
  \label{eq:non-oscillatory}
  \begin{split}
   & (2\pi h)^{- (d - 1)}\int e^{i (\sphvar - \sphvar')\cdot \dspharv/h + iG(\sphvar, \dspharv)/h} a(\sphvar, \sphvar', \dspharv, h)
    (1 - \wt{\chi}(\sphvar, \sphvar', \dspharv (\epsilon h)^{1/(\alpha - 1)}, h) )
    d\dspharv \\
 & \qquad - (2\pi h)^{- (d - 1)} \int e^{i (\sphvar - \sphvar')\cdot \dspharv/h }(1 - \wt{\chi}(\sphvar, \sphvar', \dspharv
    (\epsilon h)^{1/(\alpha - 1)}, h)) d\dspharv.
  \end{split}
\end{equation}
This requires some explanation.  To compute the composition we must
compose an operator whose Schwartz kernel is an oscillatory integral
as in \eqref{eq:Sh-at-fiber-infinity-y-dependent}, call it $I_h(y, y')$ with an operator
whose Schwartz kernel is an oscillatory integral of the form
\eqref{eq:microlocal-function-of-Lapl}.  This is done by arguing along
the lines in Appendix \ref{sec:composition}, where in particular we
see that the composition of two such oscillatory integrals is given by
$\int I_h(y, y'') \wt{I}_h(y'', y') |dy''|$.  The situation here is substantially simpler
since the operator on the right is a semiclassical pseudo, and the
expression above is obtained easily from stationary phase. (Again, see Appendix
\ref{sec:composition}.)

Setting $\wt{\dspharv} = \dspharv  (\epsilon h)^{1/(\alpha - 1)}$, $\wt{h} = (\epsilon h)^{1/(\alpha - 1)}
  h$ and
factoring gives 
\begin{equation}
  \label{eq:change-of-variables}
   (2\pi \wt{h})^{- (d - 1)} \int e^{i (\sphvar -
     \sphvar')\cdot \wt{\dspharv} / \wt{h} } \wt{a}(\sphvar, \sphvar', \wt{\dspharv}, \wt{h})    d\wt{\dspharv},
\end{equation}
where, letting $h$ remain as a function of $\wt{h}$ and $\epsilon$ for the moment,
\begin{equation*}
  \begin{split}
    \wt{a}(\sphvar, \sphvar', \wt{\dspharv}, \wt{h}) &= \lp \exp( iG(\sphvar, \wt{\dspharv} (\epsilon
    h)^{-1/(\alpha - 1)})/h) b(\sphvar, \sphvar', \wt{\dspharv} (\epsilon
    h)^{-1/(\alpha - 1)}, h) - 1 \rp \\
    & \qquad \times (1 - \wt{\chi}(\sphvar, \sphvar', \wt{\dspharv}, h)
    ).
  \end{split}
  \end{equation*}
The immediate effect of this change of variables is that
\begin{equation} \label{eq:F}
  \begin{split}
    F &= \exp( iG(\sphvar, \wt{\dspharv} (\epsilon h)^{-1/(\alpha - 1)})/h)
    \times (1 - \wt{\chi}(\sphvar, \sphvar', \wt{\dspharv}, h)) 
  \end{split}
\end{equation}
satisfies $F  - (1 -
\wt{\chi}) \in \epsilon S^{1 -
  \alpha}$, i.e.\ it is a semiclassical symbol of order $1 - \alpha$
as a function of $\wt{\dspharv}$ and $\wt{h}$ (or $\wt{\dspharv}$ and $h$ for
that matter.)
Indeed, recalling that support of $1 - \wt{\chi}$ is contained in $\{
\wt{\dspharv} \ge 1 \}$ and thus $|\wt{\dspharv}|^{-\delta}( 1 -\wt{\chi})$ is
bounded below for any $\delta > 0$, we claim first that, notation as
in \eqref{eq:G-decomposition},
$$
(1 - \wt{\chi})\wt{G}/ h = (1 - \wt{\chi}) (g(\sphvar, \wt{\dspharv})
\epsilon |\wt{\dspharv}|^{1 - \alpha} + \wt{g}(\sphvar, \sphvar', \wt{\dspharv}(\epsilon h)^{-1/(\alpha -
  1)})/h)
$$   
satisfies that $(1 - \wt{\chi})\wt{g} /h \in \epsilon S^{1 - \alpha -
  \epsilon}$ as a function of $\wt{\dspharv}$.  Indeed,
\begin{equation}
  \label{eq:symbol-estimates-2}
  \begin{split}
|    \p_{\sphvar, \sphvar'}^\alpha \p_{\wt{\dspharv}}^\beta (1 - \wt{\chi})\wt{g} /h |
\le c h^{-1}   | \sum_{\alpha' \le \alpha, \beta' \le \beta} \p_{\sphvar,
  \sphvar'}^{\alpha - \alpha'} \p_{\wt{\dspharv}}^{\beta - \beta'}(1 -
\wt{\chi}) \p_{\sphvar,
  \sphvar'}^{\alpha } \p_{\wt{\dspharv}}^{\beta}\wt{g} |,
  \end{split}
\end{equation}
and while for $\beta' \neq 0, \p_{\wt{\dspharv}}^{\beta - \beta'} (1 -
\wt{\chi})$ is compactly supported in $\wt{\chi}$, the symbol
estimates for $\wt{g}$ give that for $\beta \neq 0$
\begin{equation}\label{eq:change-of-variables-symbol-estimate}
  \begin{split}
    |(1 - \wt{\chi}) \p_{\wt{\dspharv}}^{\beta}\wt{g}| &\le (1 -
    \wt{\chi}) (\epsilon
    h)^{1/(1 - \alpha)} \la \wt{\dspharv} (\epsilon h)^{-1/(\alpha - 1)}
    \ra^{1 - \alpha - \epsilon - |\beta|} \\
    &\le (1 - \wt{\chi}) (\epsilon
    h)^{1/(1 - \alpha)  -(1 - \alpha - \epsilon - |\beta|) /(\alpha - 1)}\lp |\wt{\dspharv} 
    |^2 + (\epsilon h)^{2/(\alpha - 1)} \rp^{(1 - \alpha - \epsilon -
      |\beta|) / 2}\\
    &\le (\epsilon
    h) \la \wt{\dspharv} 
    \ra^{1 - \alpha - \epsilon - |\beta|},
  \end{split}
\end{equation}
where in the last line we used that $1 - \wt{\chi}$ is supported in
$|\wt{\dspharv}| \ge 1$, while for $\beta = 0$
firstly that
$$
|(1 - \wt{\chi}) \wt{g}(\sphvar, \sphvar', (\epsilon h)^{-1/(\alpha -
  1)}/h| \le C (1 - \wt{\chi}) \la \wt{\dspharv} (\epsilon h)^{-1/(\alpha -
  1)} \ra^{1 - \alpha - \epsilon} < C \epsilon.
$$
 The estimates in \eqref{eq:symbol-estimates-2} and
\eqref{eq:change-of-variables-symbol-estimate} together show $(1 -
\wt{\chi})\wt{G}/ h$ and thus $F - (1 - \wt{\chi})
\in \epsilon S^{1 - \alpha}$.
Moreover, 
$$
a(\sphvar, \sphvar', \wt{\dspharv} (\epsilon h)^{-1/(\alpha -
  1)}, h) = 1 + b(\sphvar,
\sphvar', \wt{\dspharv} (\epsilon h)^{-1/(\alpha - 1)}, h)
$$
remains a symbol of order $0$ and $b$ a symbol of order
$1-\alpha$, and as in the case of $F$, $b(\sphvar,
\sphvar', \wt{\dspharv} (\epsilon h)^{-1/(\alpha - 1)}, h) \in \epsilon S^{1 -
  \alpha}$ in $\wt{\dspharv}$.
Thus
$$
\wt{a} = (F - (1 - \wt{\chi})) + F b
$$
is $\epsilon$ times a semiclassical symbol of order $1 - \alpha$ whose
derivatives in $\sphvar, \sphvar'$ and $\wt{\dspharv}$ are uniformly bounded.  In
particular Calderon-Vaillancourt (see \eqref{eq:CV-norm} and below) gives
\eqref{eq:epsilon-norm-bound}.

To finish the proof, given $\epsilon$ we take $\epsilon' = \epsilon /
C$ with $C$ in \eqref{eq:epsilon-norm-bound} and use $\epsilon'$ in
the arguments above to see that in $A_h = A_{h, 1} + A_{h, 2}$ (see
\eqref{eq:Ah-bustup}), $A_{h, 1}$ has at most $c \epsilon^{- \gamma}
h^{- \alpha \gamma}$ eigenvalues at distance $\epsilon$ from $1$ while
$A_{h, 2}$ is norm bounded by $\epsilon$.  This is exactly the desired result.
\end{proof}

\subsection{Proof of Propositions~\ref{thm:trace-class} and \ref{cor:traceclass}}
 \begin{proof}[Proof of Proposition~\ref{thm:trace-class}]  Let $p \in \mathbb{N}$.  Let 
$$
A_{h}(p)  =  \{ e^{2i \beta_{h, n}} \in \spec S_h : 2^{-(p -1)} \ge
|e^{2 i
     \beta_{h, n}} - 1|  > 2^{-p} \},
$$
where elements are included with multiplicity. Then taking $\epsilon =
2^{-p}$ in \eqref{eq:away-from-1} gives
\begin{equation}
  \label{eq:4}
  |A_{h}(p) | \le c 2^{p(d - 1)/(\alpha - 1)} h^{- {\alpha/(\alpha - 1)}}.
\end{equation}
The pairing of $f$ with $\mu_h$ is given by
  \begin{equation}
    \label{eq:2}
    \begin{split}
       \la \mu_h , f \ra &= h^{{\alpha/(\alpha - 1)}} \sum_{\spec S_h} f(e^{2i \beta_{h, n}}) 
= h^{{\alpha/(\alpha - 1)}} \sum_{p = 0}^{\infty} ( \sum_{A_{h}(p)} f(e^{2i \beta_{h, n}}) ) .
      \end{split}
  \end{equation}
But,
\begin{equation}
  \label{eq:3}
  \begin{split}
   |  \sum_{A_h(p)} f(e^{2i \beta_{h, n}}) | &\le 
   \norm[w]{f}  \sum_{A_h(p)} | e^{2i \beta_{h, n}} - 1| \\
&\le \norm[w]{f} 2^{-p} |A_h(p) | \\
&\le c  \norm[w]{f} 2^{-p} (2^{p(d - 1)/(\alpha - 1)} h^{- {\alpha/(\alpha - 1)}}) \\
&\le c   \norm[w]{f} h^{- {\alpha/(\alpha - 1)}}  2^{-p + p (d - 1)/(\alpha - 1)},
  \end{split}
\end{equation}
Thus
\begin{equation}
  \label{eq:5}
  \la \mu_h , f \ra \le c   h^{\alpha/(\alpha - 1)} \norm[w]{f} h^{- {\alpha/(\alpha - 1)}}  \sum_{p =
    0}^\infty 2^{p((d - 1)/(\alpha - 1) -  1) } \le  c   \norm[w]{f} \sum_{p =
    0}^\infty 2^{p( (d - 1)/(\alpha - 1) -  1) },
\end{equation}
and the above is summable if and only if 
$$
1 > \frac{d - 1}{\alpha - 1} \iff \alpha > d,
$$
which is exactly our assumption on $\alpha$.
\end{proof}

\begin{proof}[Proof of Proposition~\ref{cor:traceclass}] It is enough to prove for $k=1$, as for any other value of $k$, we can write $S_h^k - \Id$ as the product of $S_h - \Id$ with a bounded operator. For $k=1$, Proposition~\ref{thm:rough-eigenvalue-asymptotics} shows that the number of eigenvalues $z$ (counted with multiplicity) of $S_h$ such that $|z-1| \in [2^{-j}, 2^{-j+1}]$ is bounded by $C h^{-\alpha \gamma} 2^{j\gamma}$. Since $\gamma < 1$, we can sum $2^{-j+1} \times C h^{-\alpha \gamma} 2^{j\gamma}$ over $j \in \NN$, which is a bound for the sum of $|z-1|$ over all eigenvalues $z$. It follows that the trace norm of $S_h - \Id$ is finite. 
\end{proof}


\section{Trapping energies}\label{sec:trapping}
Suppose now that $E$ is a trapping energy for the potential $V$. In this case, we write the scattering matrix $S_h(E)$ as the scattering matrix $\tilde S_h(E)$ for a different potential $\tV$, which is nontrapping at energy $E$, plus a small remainder.  We can choose the potential $\tV$ to be equal to $V$ near infinity. To do this, we first choose a function $\phi \in C_c^\infty(\RR_+)$, equal to $1$ in a neighbourhood of $0$, and monotone nonincreasing. Then $\tV := V + 2E \phi(|x|/R)$ will be nontrapping at energy $E$, for sufficiently large $R$. 

We then express the scattering matrix $S_h(E)$ in terms of $\tilde S_h(E)$. 
To do this, we follow \cite[Section 8B]{GHS}. Let $R_h = (h^2\Delta +
V - (E+i0))^{-1}$ and $\tilde R_h = (h^2\Delta + \tV - (E+i0))^{-1}$
be the outgoing resolvents for the unperturbed and perturbed
potential, respectively. Also, let $\chi_i$, $i = 1, 2, 3$ be cutoff
functions supported near infinity in $\RR^n$, equal to $1$ for $|x|
\geq 2R$ and $0$ for $|x| \leq R$, such that $\chi_i \chi_j = \chi_j$
when $j < i$. Then, following the derivation of \cite[Equation
(8-7)]{GHS}, i.e.\ taking $\textsc{H} = \Delta + V$ and $\lambda =
E/h^2$ in that equation, we obtain 
\begin{equation}
\chi_2 R_h \chi_1 = \chi_2 \tilde R_h \chi_1 + \chi_2 \tilde R_h [\chi_3, h^2 \Delta + V] R_h [h^2 \Delta + V, \chi_2] \tilde R_h \chi_1.
\label{res-identity}\end{equation}
These Schwartz kernels are defined on $\RR^d_x \times \RR^d_{x'} \times (0, h_0]_h$. 
As discussed in \cite{HW2008}, if we use polar coordinates $x = (r,
\omega)$, $x' = (r', \omega')$, multiply the kernel of $R_h$ by
${r'}^{(d-1)/2}$ and take the limit $r' \to \infty$, we obtain the
Poisson kernel $P_h(E)$, which is a function of $(x, \omega', h)$. If
we then multiply the Poisson kernel $P_h(E)$ by $r^{(d-1)/2}$ and take
the (distributional) limit $r \to \infty$, we obtain the kernel of the
absolute scattering matrix; multiplying by $i^{(d-1)/2}$ and composing
with the antipodal map $A$, we obtain the scattering matrix $S_h(E)$
as we have normalized it.  The same operations applied to $\tilde R_h$
produce $\tilde P_h(E)$ and $\tilde S_h(E)$. Applying these operations
to \eqref{res-identity}, we obtain 
\begin{equation}
S_h(E)  = \tilde S_h(E) + i^{(d-1)/2} A \tilde P^*_h(E) [\chi_3, h^2 \Delta + V] R_h [h^2 \Delta + V, \chi_2] \tilde P_h(E) .
\label{sm-identity}\end{equation}
For brevity, we write this in the form 
\begin{equation}
S_h(E)  = \tilde S_h(E) + B_h(E) ; 
\end{equation}
clearly $B_h(E)$ is a uniformly bounded family of operators on $L^2(\mathbb{S}^{d-1})$. 

Our previous arguments apply to $\tilde S_h(E)$, since $E$ is a nontrapping energy for $\tV$. So it suffices to show that the perturbation $B_h(E)$ has no effect on the weak-$*$ limit $\tilde \mu$ of the measures $\tilde \mu_h$ associated to $\tilde S_h(E)$, as $h \to 0$. 

To show this, we now cut off to small and large frequencies using a cutoff $\chi(h^2 \Delta_{\mathbb{S}^{d-1}})$, where $\chi(t)$ is compactly supported, and identically $1$ near $t=0$. For simplicity we write this operator simply as $\chi$. Thus we write
\begin{equation}
S_h(E) = \chi \tilde S_h(E) + \chi B_h(E) + (\Id - \chi) \tilde S_h(E) + (\Id - \chi) B_h(E). 
\label{S-decomp}\end{equation}
The first term is an FIO with compact microsupport, hence has trace norm bounded by $C h^{-(d-1)}$. The second term also has trace norm bounded by $C h^{-(d-1)}$, since this is true of $\chi$ which is also an FIO with compact microsupport. The third term is the principal term,  and the fourth we bound using wavefront set results. 
In fact, according to \cite{HW2008}, the semiclassical wavefront set of $\tilde P^*_h(E)$ is contained in 
$$
\{ (\omega, \eta; x, \xi) \mid \text{ the bicharacteristic through } (x, \xi) \text{ has asymptotic } t \mapsto \eta + \omega (t - t_0), t \to \infty \}
$$
when the point $x$ is restricted to a fixed compact set. Now consider the composition $(\Id - \chi) A \tilde P^*_h(E) [\chi_3, h^2 \Delta + V]$. Composition on the right with $[\chi_3, h^2 \Delta + V]$ restricts the wavefront set to points $x \in \supp \nabla \chi_3$, that is, to $x$ lying in some fixed compact set in $\RR^d$. On the other hand, composition on the left with $(\Id - \chi)$ restricts the wavefront set to points $(\omega, \eta)$ in the support of the symbol of $\chi$. 
By choosing $\chi$ suitably, we can arrange that this support is contained in $|\eta| \geq R'$ for $R'$ arbitrary. By choosing $R'$ sufficiently large, we arrange that the wavefront set of  $(\Id - \chi) A \tilde P^*_h(E) [\chi_3, h^2 \Delta + V]$ vanishes. That implies that the Schwartz kernel of this operator is smooth and $O(h^\infty)$. The trace norm of the fourth term in \eqref{S-decomp} is therefore $O(h^\infty)$. 

Now consider all the terms in $S_h(E)^k - \Id$, where $S_h(E)$ is decomposed according to \eqref{S-decomp}. The main term, $\big( (\Id - \chi) \tilde S_h(E) \big)^k - \Id$, is treated as in Sections~\ref{sec:the-scattering-matrix} -- \ref{sec:eigenvalue-distribution}. All other terms have trace norm bounded by $O(h^{-(d-1)})$, and therefore their contribution to 
$h^{\gamma \alpha} \Trace (S_h(E)^k - \Id)$ vanishes in the limit $h \to 0$. 
This completes the proof of the Main Theorem in the case of a trapping energy.


\begin{appendix}

\section{Regularity of the sojourn map}\label{sec:soujourn}

In this appendix we prove Proposition~\ref{thm:sojourn-near-infty} and
Lemma~\ref{lem:S-structure}. Our first task is to determine the
regularity of the Legendre submanifold $L$ \eqref{eq:total-sojourn-relation} as $|\eta| \to
\infty$. To do this, we use the fact that $L$ is the boundary value of
a Legendre submanifold $\SR$ over a space of dimension one greater which is a
bicharacteristic flowout, that is, the union of bicharacteristic
rays. We start by defining some spaces of conormal functions, and then
proceed to describe $\SR$ and its ambient contact manifold.

This process will use the language, developed by Melrose \cite{damwc, tapsit}, of analysis on manifolds with corners.  Though some of
this is quite involved we will provide some brief explanations and
definitions for the
convenience of the reader.

\subsection{Conormal regularity of solutions to ODEs}\label{subsec:conormal}

Let $M$ be a manifold with corners, with boundary hypersurfaces
$H_1,\dots, H_m$ and boundary defining functions $\rho_1, \dots,
\rho_m$ respectively \cite{damwc}.  Thus the boundary $\p M$ is equal
to the union of the $H_i$, and for each $i$, $\rho_i$ is a
non-negative, smooth function on $M$ with $H_i = \{ \rho_i = 0\}$ and $d \rho_i \neq 0$ on $H_i$. Let $\rho = \rho_1 \dots, \rho_m$
be the product
of boundary defining functions.  We say that a vector field
$\mathcal{V}$ on $M$ is a b-vector field if it is smooth, and tangent
to each boundary hypersurface $H_i$, or equivalently if
$\mathcal{V}(\rho_i) = O(\rho_i)$ for each $i$.  

Let $\underline{\epsilon} = (\epsilon_1, \dots, \epsilon_m) \in \mathbb{R}^m$ be a multiweight, one for each boundary hypersurface of $M$. The ($L^\infty$-based) space of conormal functions with weight $\underline{\epsilon}$, $\mathcal{A}^{\underline{\epsilon}}(M)$, is defined as follows:
\begin{equation}
\mathcal{A}^{\underline{\epsilon}}(M) = \left\{ f \in
  C^\infty(M^\circ) \mid 
  \begin{array}{c}
\rho^{-\underline{\epsilon}} f \in L^\infty(M) \mbox{ and } \rho^{-\underline{\epsilon}} \mathcal{V}_1
  \dots \mathcal{V}_k f \in L^\infty(M) \\
 \text{ for any $k$ b-vector
    fields } \mathcal{V}_1, \dots, \mathcal{V}_k     
  \end{array}
\right\}.
\end{equation}
Here $\rho^{\underline{\epsilon}}$ is shorthand notation for the product $\rho_1^{\epsilon_1} \dots \rho_m^{\epsilon_m}$. 
That is, $f \in \rho^{\underline{\epsilon}}L^\infty(M)$, and remains
in this space under repeated differentiations by b-vector fields on
$M$.  The space $\mathcal{A}^{\underline{\epsilon}}$ is a Frechet
space whose metric we describe below for a simple example.

We also use the notation 
$$C^{\infty, \underline{\epsilon}}(M) = C^\infty(M) +
\mathcal{A}^{\underline{\epsilon}}(M).
$$ 
The regularity condition of our potential $V$ in the Main
Theorem can be phrased in terms of the above spaces; the assumption on $V$ can be expressed in terms of the
radial compactification $\overline{\RR^d}$ of $\RR^d$, where $1/r$ is
taken as the boundary defining function at the `sphere at
infinity,' and is equivalent to assuming that for some $0 < \epsilon < 1$, $r^\alpha V \in C^{\infty,
  \epsilon}(\overline{\RR^d})$.  (Equivalently, $V \in r^{-\alpha}C^{\infty,
  \epsilon}(\overline{\RR^d})$.) We abuse notation slightly by
defining a smooth map $u = (u_1, \dots,
u_n)$  on $M^\circ$ with values in $\mathbb{C}^n$ to lie in
$\mathcal{A}^{\underline{\epsilon}}(M)$ if and only if its components
do. 

We note, for later use, the following result. The proof is straightforward and omitted. 

\begin{lemma}\label{lem:inverses} \

(i) If $f \in \Cinfep(M)$ is bounded away from zero, then $1/f \in \Cinfep(M)$.

(ii) If $S: M \to N$ is a b-map\footnote{This means that the inverse
  image of every boundary defining function on $N$ is a product of
  boundary defining functions on $M$, times a smooth non-vanishing function. An invertible b-map induces, in particular, a bijection between the codimension $k$-faces of $M$ and the codimension $k$-faces of $N$.} between manifolds with corners $M$ and $N$ such that all components of $S$ have regularity $\Cinfep(M)$, and $S$ is invertible in the sense that it is invertible as a map and its Jacobian determinant is bounded away from zero, then the inverse map has regularity $\Cinfep(N)$. 

(iii) Let $\gamma_1, \dots, \gamma_m$ be positive exponents, and suppose that $\epsilon > 0$ is sufficiently small (relative to the $\gamma_i$). Then the statements (i) and (ii) above also hold if the space $\Cinfep$ is replaced by $C^\infty + \Pi \rho_i^{\gamma_i} \Cinfep$. 
\end{lemma}

It is well known that solutions of ODEs 
$$
\frac{dy}{dt} = F(y,t), \quad y(0) = y_0
$$
are smooth if $F$ is smooth, and $y$ also depends smoothly on the
initial condition $y_0$. See e.g.\ Hartman \cite[Chapter 5]{Hartman}.
Here, we note the following variant of this standard result. We find it convenient
to write the ODE in terms of a b-derivative, $t \partial_t$.

\begin{proposition}\label{prop:conormal-regularity}
Consider the ODE
\begin{equation}\begin{aligned}
t \frac{dz}{dt} &= F(z, s, t) \qquad \ z(0) = z_0,  \\
t \frac{ds}{dt} &= s G(z, s, t) \qquad s(0) = s_0 > 0.
\end{aligned}\label{ode}\end{equation}
for $z \in \RR^p$ and $s \in \RR_+$. 

(i) Suppose that $F, G \in \mathcal{A}^{\beta_1, \beta_2}(\RR^{p}_{z} \times \RR^+_{s}
\times \RR^+_t)$, $\beta_i > 0$, where  $\beta_1$ refers to the
$s$ variable and $\beta_2$ to $t$. Then the solution $z = z(z_0, s_0,
t), s = s(z_0, s_0, t)$ satisfies 
$$
z(z_0, s_0, t) - z_0, \quad \frac1{s_0} (s(z_0, s_0, t) - s_0) \in 
\mathcal{A}^{\beta_1, \beta_2}(\RR^{p}_{z_0}
\times \RR^+_{s_0} \times \RR^+_t)$$ 
locally near $t = 0$. 

(ii) Suppose that $F, G \in t C^\infty(\RR^{p}_{z} \times \RR^+_{s}
\times \RR^+_t) +  \mathcal{A}^{\beta_1, \beta_2}(\RR^{p}_{z} \times \RR^+_{s}
\times \RR^+_t)$, $\beta_i > 0$. Then  the solution $z = z(z_0, s_0,
t), s = s(z_0, s_0, t)$ satisfies 
$$
z - z_0,  \frac1{s_0} (s- s_0) \in 
t C^\infty(\RR^{p}_{z_0} \times \RR^+_{s_0}
\times \RR^+_t) +  \mathcal{A}^{\beta_1, \beta_2}(\RR^{p}_{z_0} \times \RR^+_{s_0}
\times \RR^+_t)$$ 
locally near $t = 0$. 

(iii) Let $\beta_i = \gamma_i + \epsilon$, where $\gamma_i > 0$ and $\epsilon$ is sufficiently small. Suppose that 
$F, G \in t C^\infty(\RR^{p}_{z} \times \RR^+_{s}
\times \RR^+_t) + s^{\gamma_1} t^{\gamma_2} C^\infty(\RR^{p}_{z} \times \RR^+_{s}
\times \RR^+_t) +  \mathcal{A}^{\beta_1, \beta_2}(\RR^{p}_{z} \times \RR^+_{s}
\times \RR^+_t)$. Then  the solution $z = z(z_0, s_0,
t), s = s(z_0, s_0, t)$ satisfies 
$$
z - z_0,  \frac1{s_0} (s- s_0) \in 
t C^\infty(\RR^{p}_{z_0} \times \RR^+_{s_0}
\times \RR^+_t) + s_0^{\gamma_1} t^{\gamma_2} C^\infty(\RR^{p}_{z_0} \times \RR^+_{s_0}
\times \RR^+_t) +  \mathcal{A}^{\beta_1, \beta_2}(\RR^{p}_{z_0} \times \RR^+_{s_0}
\times \RR^+_t)
$$
 locally near $t = 0$. 

\end{proposition}

\begin{proof} 
We start by making some reductions. We first let $\tilde z = z - z_0$ and $\tilde s = \log(s/s_0)$. 
Then $\tilde z(0)$ and $\tilde s(0)$ solve the initial value problem 
\begin{equation}\begin{aligned}
t\frac{d\tilde z}{dt} &= F(\tilde z + z_0, s_0 e^{\tilde s}, t)  \qquad \ \tilde z(0) = 0 \\
t\frac{d\tilde s}{dt} &= G(\tilde z + z_0, s_0 e^{\tilde s}, t) \qquad  \tilde s(0) = 0.\\
\end{aligned}\label{ode2}\end{equation}
Thus, we can combine $(\tilde z, \tilde s)$ into a new variable $Z$, satisfying an equation of the form
$$
t\frac{dZ}{dt} = \FFF(Z, z_0, s_0, t)  \qquad \ \tilde Z(0) = 0 \\
$$
and show conormal regularity in the $(s_0, t)$ variables. 

Let $S, T$ denote the differential operators $s_0 \partial_{s_0}$ and $t \partial_t$ respectively. 
To prove (i), we need to show that $S^j T^k D_{z_0}^\alpha Z(z_0, s_0,
t)$ is bounded by $C s_0^{\beta_1} t^{\beta_2}$ for all
$(j,k,\alpha)$. This is clear when $j = k = |\alpha|= 0$, directly
from a pointwise estimate on $\FFF$. We prove by induction on
$j+k+|\alpha|$. We find that $w := S^j T^k D_{z_0}^\alpha Z(z_0, s_0,
t)$ has $i^{th}$ component satisfying an ODE of the form
$$
t \frac{dw_i}{dt} =  \sum_j \frac{\partial \FFF_i}{\partial z_j} w_j + B,
$$
where $B$ is a sum of products of factors, each of which is a b-derivative of the form $S^{j'} T^{k'} D_{z_0}^{\alpha'}$ applied to $\FFF$ or $Z$, and where the total number of derivatives applied to any factor of $Z$ is strictly less than $j+k+|\alpha|$. By using an integrating factor, the bound $C s_0^{\beta_1} t^{\beta_2}$ on any b-derivative of $\FFF$, and the inductive assumption for lower-order b-derivatives of $Z$,  we deduce a similar bound on $S^j T^k D_{z_0}^\alpha Z$, completing the proof. 

To prove (ii), we write $\FFF = \Fsm + \FFF_c$, where $\Fsm$ is $t$ times a smooth function, and $\FFF_c$ is conormal of order $(\beta_1, \beta_2)$. We write $\Zsm$ for the solution to the ODE
\begin{equation}
t \frac{d \Zsm}{dt} = \Fsm(\Zsm, z_0, s_0, t).
\label{Zsm}\end{equation}
Then $\Zsm \in t C^\infty$ using standard ODE theory. So consider $Z - \Zsm$. This satisfies the ODE
\begin{equation}
t \frac{d (Z - \Zsm)}{dt} = \Fsm(Z, z_0, s_0, t) - \Fsm(\Zsm, z_0, s_0, t) + \FFF_c(Z, z_0, s_0, t).
\label{ZZsm}\end{equation}

It suffices to show that $Z - \Zsm$ is conormal of order $(\beta_1, \beta_2)$. We prove, by induction on $j+k+|\alpha|$, that $S^j T^k D_{z_0}^\alpha (Z - \Zsm)$ is bounded by $C s_0^{\beta_1} t^{\beta_2}$. When $j+k+|\alpha|=0$, notice that the RHS of \eqref{ZZsm} is bounded by $C |Z - \Zsm| + C s_0^{\beta_1} t^{\beta_2}$. We conclude, using an integrating factor, that $|Z - \Zsm|$ is bounded by $Cs_0^{\beta_1} t^{\beta_2}$. Now consider the b-differential operator $S^j T^k D_{z_0}^\alpha$ applied to $Z - \Zsm$. The argument is similar to part (i). 
Consider the ODE satisfied by $S^j T^k D_{z_0}^\alpha (Z - \Zsm)$. On the RHS there will be a sum of products of factors of various sorts. These terms must of one of the following type. The first type is
\begin{equation*}\begin{gathered}
\sum_j \Big( \frac{\partial \Fsm_i(Z)}{\partial z_j}  S^j T^k D_{z_0}^\alpha Z_j - \frac{\partial \Fsm_i(\Zsm)}{\partial z_j}  S^j T^k D_{z_0}^\alpha \Zsm_j \Big) \\
= \sum_j \Big( \frac{\partial \Fsm_i(Z)}{\partial z_j}  S^j T^k D_{z_0}^\alpha (Z_j - \Zsm_j)  + \Big( \frac{\partial \Fsm_i(Z)}{\partial z_j} - \frac{\partial \Fsm_i(\Zsm)}{\partial z_j} \Big) S^j T^k D_{z_0}^\alpha \Zsm_j .
\end{gathered}\end{equation*}
Notice that the first term is a bounded multiple of $S^j T^k D_{z_0}^\alpha (Z_j - \Zsm_j)$, while the second is bounded in magnitude by $C |Z - \Zsm|$, and hence by $Cs_0^{\beta_1} t^{\beta_2}$. 

The next type are terms that involve lower-order b-derivatives of $Z$ and $\Zsm$. All such terms include a factor that is either of the form  $S^{j'} T^{k'} D_{z_0}^{\alpha'} (Z_j - \Zsm_j)$ or $(S^{j'} T^{k'} D_{z_0}^{\alpha'} \Fsm)(Z) - (S^{j'} T^{k'} D_{z_0}^{\alpha'} \Fsm)(\Zsm)$, or else involve $\FFF_c$. Using the inductive assumption, this gives an ODE of the form 
$$
t \frac{dw_i}{dt} =  \sum_j \frac{\partial \FFF_i}{\partial z_j} w_j + B,
$$
for $S^j T^k D_{z_0}^\alpha (Z - \Zsm)$, where $B$ is bounded by $Cs_0^{\beta_1} t^{\beta_2}$. As in part (i), we conclude that $S^j T^k D_{z_0}^\alpha (Z - \Zsm)$ is bounded by $C s_0^{\beta_1} t^{\beta_2}$. 

The proof of part (iii) is similar to part (ii). We write $\FFF = \Fsm + s_0^{\gamma_1} t^{\gamma_2} \FFF_\gamma + \FFF_\beta$, where $\Fsm$ and $\FFF_\gamma$ are smooth. We first find a function $\ZZZ(z_0, s_0, t)$ that solves 
\begin{equation}
t \frac{d\ZZZ}{dt} = \Fsm(\ZZZ, z_0, s_0, t) + s_0^{\gamma_1} t^{\gamma_2} \FFF_\gamma(\ZZZ, z_0, s_0, t)
\label{conormalODE}\end{equation}
up to an error which is conormal of order $(\beta_1, \beta_2)$. To do this, we start from the solution $\Zsm$ of \eqref{Zsm}, and modify it in order to solve away the term $s_0^{\gamma_1} t^{\gamma_2} Z_\gamma(z, z_0, s_0, t)$ to leading order, both at $s_0=0$ and at $t=0$. We propose an ansatz of the form $\ZZZ = \Zsm +s_0^{\gamma_1} t^{\gamma_2} Z_\gamma(z, z_0, s_0, t)$, where $Z_\gamma$ is $C^\infty$. Let $v(z_0, t)$ be the restriction of $Z_\gamma$ to $s_0=0$, and $w(z_0, s_0)$ be the restriction to $t=0$. To simplify notation, we shall suppress the dependence of  all quantities on $z_0$ from now on. 

To see what the functions $v$ and $w$ must be, we substitute $\Zsm +
s_0^{\gamma_1} t^{\gamma_2} Z_\gamma(s_0, t)$ into the ODE. This gives
a polyhomogeneous expansion both as $s_0 \to 0$ and $t \to 0$, with
the first possible non-integral power $s_0^{\gamma_1}$ as $s_0 \to 0$ and $t^{\gamma_2}$ as $t \to 0$. We seek to make these powers agree on the LHS and RHS of the ODE; this will determine $v$ and $w$ uniquely. 

Computing the $s_0^{\gamma_1}$ terms of the RHS and LHS of (A.7) and setting them equal gives
\begin{equation}
t^{\gamma_2} \Big( t \frac{dv_i}{dt} + \gamma_2 v_i \Big) = t^{\gamma_2} \Big( \sum_j \frac{\partial \Fsm(\Zsm(0, t), 0, t)}{\partial z_j} v_j + {\FFF_\gamma}_i(\Zsm(0, t), 0, t) \Big).
\end{equation}
Dividing by $t^{\gamma_2}$ gives an ODE for $v_i$ which has a smooth solution. Moreover, since $\Fsm = O(t)$, the the value of $v_i$ at $t=0$ is given by 
\begin{equation}
v_i(0) = \gamma_2^{-1} {\FFF_\gamma}_i(z_0, 0, 0). 
\end{equation}

Similarly, the coefficient of $t^{\gamma_2}$ of the expansion at $t=0$ of the ODE is given by 
\begin{equation}
s_0^{\gamma_1}  \gamma_2 w_i(s_0, 0) = s_0^{\gamma_1} \Big(  \sum_j  \frac{\partial \Fsm(\Zsm(s_0,0), s_0, 0)}{\partial z_j} w_j + {\FFF_\gamma}_i(\Zsm(s_0,0), s_0, 0) \Big) .
\end{equation}
Clearly this has a smooth solution $w_i(s_0)$, with $w_i(0) = \gamma_2^{-1}  {\FFF_\gamma}_i(z_0, 0, 0) = v_i(0)$. Since $v(0) = w(0)$, we can find a smooth $Z_\gamma(s_0, t)$ that agrees with $v$ at $s_0=0$ and with $w$ at $t=0$. Then it is easy to check that $\Zsm +s_0^{\gamma_1} t^{\gamma_2} Z_\gamma(z, z_0, s_0, t)$ solves the ODE \eqref{conormalODE} up to an error term that is conormal of order $(\beta_1, \beta_2)$, provided that $\epsilon$ is sufficiently small. 

To complete the proof, we look for a solution $Z'(z_0, s_0, t)$ of the ODE 
$$
t\frac{dZ'}{dt} = \Big( \Fsm + s_0^{\gamma_1} t^{\gamma_2} \FFF_\gamma + \FFF_\beta \Big)(Z, z_0, s_0, t)
$$
of the form $\Zsm +s_0^{\gamma_1} t^{\gamma_2} Z_\gamma(z, z_0, s_0, t) + Z_\beta$. It suffices to show that $Z_\beta$ is conormal of order $(\beta_1, \beta_2)$. This is proved using exactly the same argument as in (ii) above, so we omit the details. 
\end{proof}

In the course of this proof, we have essentially proved the following perturbation result: 

\begin{lemma}\label{lem:comparison} Suppose that $z, s$ solve the ODE \eqref{ode}, where $F, G \in tC^\infty$. Let $\tilde F, \tilde G$ be functions in $s^{\gamma_1} t^{\gamma_2} C^\infty + \mathcal{A}^{(\beta_1, \beta_2)}$,  where $\gamma_i$ and $\beta_i$ are as in Proposition~\ref{prop:conormal-regularity}, part (iii), and let $F_* = F + \tilde F$ and $G_* = G + \tilde G$. 

Let   $z_*$, $s_*$ solve the ODE with $F, G$ replaced with $F_*, G_*$, and with the same initial conditions as in \eqref{ode}. Then  
$$
z(t) - z_*(t), \frac1{s_0}(s - s_*) \in s_0^{\gamma_1} t^{\gamma_2} C^\infty + \mathcal{A}^{(\beta_1, \beta_2)}. 
$$

\end{lemma}

The last result we shall need is closely related related to Proposition~\ref{prop:conormal-regularity}, but where the initial conditions are specified at a positive value of $t$, say $t=\delta$, where we suppose that $\delta > 0$ is sufficiently small that the solution exists on the time interval $t \in [0, \delta]$, and we are interested in the value at $t=0$. To state these results we need to introduce spaces of functions with different sorts of regularity in the $s$ and the $t$ variable. We write $\mathcal{A}^\beta_s C^\infty_{t, z}$ for the space of functions with conormal regularity of order $\beta$ in the $s$ variable and $C^\infty$ regularity in $t$ and $z$. 

\begin{proposition}\label{prop:cr2} Let $(z, s)$ solve the ODE 
\begin{equation}\begin{aligned}
t \frac{dz}{dt} &= F(z, s, t) \qquad \ z(1) = z_0,  \\
t \frac{ds}{dt} &= s G(z, s, t) \qquad s(1) = s_0 > 0
\end{aligned}\label{ode3}\end{equation}
with initial conditions now at $t=1$. Then

(i) Suppose that $F, G$ are as in (i) of Proposition~\ref{prop:conormal-regularity}. Then  
$$
z(z_0, s_0, t) - z_0, \quad \frac1{s_0} (s(z_0, s_0, t) - s_0) \in \mathcal{A}^{\beta_1}_{s_0}  C^\infty_{t,z}
+ \mathcal{A}^{\beta_1, \beta_2}(\RR^{p}_{z_0}
\times \RR^+_{s_0} \times \RR^+_t).$$ 

(ii) Suppose that $F, G$ are as in (ii) of Proposition~\ref{prop:conormal-regularity}. Then  
$$
z - z_0,  \frac1{s_0} (s- s_0) \in C^\infty + \mathcal{A}^{\beta_1}_{s_0}  C^\infty_{t,z} +
  \mathcal{A}^{\beta_1, \beta_2}(\RR^{p}_{z_0} \times \RR^+_{s_0}
\times \RR^+_t).$$ 

(iii) Suppose that 
$F, G$ are as in (iii) of Proposition~\ref{prop:conormal-regularity}. Then   
$$
z - z_0,  \frac1{s_0} (s- s_0) \in C^\infty + s_0^{\gamma_1} (C^\infty + t^{\gamma_2} C^\infty) + 
\mathcal{A}^{\beta_1}_{s_0} \Big( C^\infty_{t,z} + t^{\gamma_2} C^\infty_{t,z}  + \mathcal{A}^{\beta_2}_t C^\infty_z \Big). 
$$

\end{proposition}
The proof is essentially identical to that of Proposition~\ref{prop:conormal-regularity}, and so is omitted. The only difference is that, instead of integrating from $t=0$, we integrate from $t=1$, so that, for example, when we integrate a term of the form $t^{\gamma_2} C^\infty_t$ with respect to $dt/t$, we only get $C^\infty_t + t^{\gamma_2} C^\infty_t$ rather than just $t^{\gamma_2} C^\infty_t$, accounting for the extra terms in Proposition~\ref{prop:cr2} compared to Proposition~\ref{prop:conormal-regularity}. 

A simple consequence of Proposition~\ref{prop:cr2} is 
\begin{corollary} \label{cor:solat0}
In case (iii) of Proposition~\ref{prop:cr2}, the functions $z(0), s(0)/s_0$ are $C^\infty + s_0^{\gamma_1} \Cinfe$ functions of the initial data $(z_0, s_0)$. 
\end{corollary}

\subsection{The sojourn relation}\label{sec:sojourn-relation}

As we describe concretely in the following subsection, according to \cite{HW2008}, the Poisson operator is a microlocal
object associated to the `sojourn relation' $\SR$.  We now proceed to
describe $\SR$ and determine its regularity properties.  
In this paper, we shall take the viewpoint that $\SR$ is a Lagrangian
submanifold of $T^* \RR^d \times T^* \mathbb{S}^{d-1}$, that extends
nicely to a certain compactification of this space\footnote{In
  \cite{HW2008}, the sojourn relation was viewed as a Legendre
  submanifold of a space with one extra dimension, with the extra
  coordinate being the variable denoted $\phi$ below. Here, we take
  the view that $\phi$ is a function defined on $\SR$.}.

First we describe this (partial) compactification\footnote{Our partial
  compactification serves to make the energy surface $\{ |\xi|^2 + V =
  E \}$ compact, which is all that matters.} of $T^* \RR^d \times T^* \mathbb{S}^{d-1}$. Let $r$ denote the radial variable $|x|$ and let $y$ be local coordinates on $\mathbb{S}^{d-1}$. We write $h^{ij}(y)$ for the (dual) metric on $\mathbb{S}^{d-1}$ with respect to these local coordinates. If we write $(\lambda, \eta)$ for cotangent variables dual to $(r, y)$ on $\RR^d$, then it is natural to use $(\lambda, \mu = \eta/r)$ near infinity, as these are variables that are homogeneous of degree zero under dilations, i.e.\ remain of fixed length as $r \to \infty$.  We write $\eta'$ for a cotangent variable dual to $y'$ on $T^* \mathbb{S}^{d -1}$, and scale it in the same way as $\eta$; that is, let $\mu' = \eta'/r$. 
Finally, we radially compactify Euclidean space by introducing $\rho = r^{-1}$ and adding a boundary at $\rho = 0$. (However, the space is still not compact as $\mu, \mu'$ vary in $\RR^{d-1}$ and $\phi$ varies in $\RR$.)

\begin{equation}
\mbox{We denote the space with coordinates  } (\rho, y, y'; \lambda, \mu,
\mu', \phi)\label{coords2} \mbox{ by $\X$.}
\end{equation}
A more invariant description of this space is
given in \cite{HW2008}, but we wish to avoid the geometric intricacies
here. The space $\X$ is a manifold with boundary, with the boundary
defined by $\{ \rho = 0 \}$, and we ignore the apparent singularity in
$\rho$ at $r = 0$ as we work in a neighborhood of $\rho = 0$.

The space $\X$ (or at least its interior) is a symplectic manifold with contact form $d\xi_j \wedge dx_j + d\eta'_i \wedge dy_i$. 
Let $\Vec$ be the Hamilton vector field for the Hamiltonian $|\xi|^2 +
V(r, y) - E = \lambda^2 + |\mu|^2 + V - E$, and let $\Vec' = r
\Vec$. 
In the coordinates $(r, y, \lambda, \eta, y', \eta')$,  $\Vec$ is given by 
\begin{equation}\begin{gathered}\label{eq:Vec}
\Vec = 2\lambda \frac{\partial}{\partial r} + \frac{2h^{ij} \eta_j}{r^{2}}
\frac{\partial}{\partial y_i} + \Big( \frac{2h^{ij} \eta_i
  \eta_j}{r^3} - \frac{\partial V}{\partial r}
\Big)\frac{\partial}{\partial \lambda} \\ - \Big(  \frac{\partial
  h^{ij}}{\partial y_k}\frac{\eta_i \eta_j}{r^2} -  \frac{\partial V}{\partial y_k} \Big)
\frac{\partial}{\partial \eta_k},
\end{gathered}\end{equation}
where we sum over repeated indices. In the coordinates $(\rho, y, \lambda, \mu, y', \mu')$, $\Vec'$ is given by 
\begin{equation}\begin{gathered}
\Vec' = -2\lambda \Big( \rho \frac{\partial}{\partial \rho} + \mu \cdot \frac{\partial}{\partial \mu} + \mu' \cdot \frac{\partial}{\partial \mu'} \Big) 
+ 2h^{ij} \mu_i \frac{\partial}{\partial y_j} \\\qquad \ + \Big(2 h^{ij} \mu_i
\mu_j +  \rho \frac{\partial V}{\partial \rho} \big)
\frac{\partial}{\partial \lambda} - \Big( \frac{\partial
  h^{ij}}{\partial y_k} \mu_i \mu_j -  \frac{\partial
  V}{\partial y_k}  \Big) \frac{\partial}{\partial \mu_k} . 
\end{gathered}\label{Vec'}\end{equation}

We now perform the operation of `blowing up' $\X$ at the submanifold
$Z = \{ \rho = 0, \mu = 0, \mu' = 0 \}$.  
This operation consists of 
replacing $Z$ with its inward pointing spherical normal bundle, which
turns the space $\X$ into a manifold with codimension 2 corners that
we shall denote $[\X; Z]$. This means essentially that $[\X; Z]$ is
the `minimal' manifold with corners on which the polar coordinates 
$$
\trho  = (\rho^2 + |\mu|^2 + |\mu'|^2)^{1/2} , \qquad \rho_B = \frac{\rho}{\trho}, \qquad \theta = (
 \frac{\mu}{\trho}, \frac{\mu'}{\trho}), 
$$
together with $y, y'$, extend smoothly up to all boundary faces.
It can be viewed as the geometric realization of
polar coordinates at $Z$, that is, the space on which polar coordinates are 
smooth. The space $[\X; Z]$ has two boundary
hypersurfaces. One boundary hypersurface is the original boundary
$\rho = 0$, or rather the lift of this to the blown up space; we shall
denote this $B$. The other is the boundary hypersurface $\tilde Z$
created by blowup. We let $\trho$ denote \textit{any} boundary defining
function for $\tilde Z$; the above formula for $\trho$ is just an
example, as any $\trho$ satisfying the properties for bdf's (see
Section \ref{subsec:conormal}) will work, and in fact when convenient
we will take $\trho = |\mu|$ near the intersection of $B$ with
$\wt{Z}$ (the `corner') and $\trho = \rho$ in the interior of $\tilde Z$.  It follows that
$\rho_B := \rho / \trho$ is a boundary defining function of $B$.
Away from $B$, coordinates near $\tilde Z$ are 
\begin{equation}
\rho, \ \eta, \ \eta', \ y, \ y', \ \lambda,
\label{coords1}\end{equation}
or equivalently one can take $(\trho, \ \eta, \ \eta', \ y, \ y', \
\lambda)$.  Indeed, both $\mu / \rho =
\eta$ and $\mu' / \rho = \eta'$ are bounded maps on compact subsets of
the interior of $\tilde Z$, and thus, not only can we take $\trho =
\rho$ but the above functions can be checked to yield a coordinate
patch on a tubular neighborhood $\tilde Z^{\circ}  \times
[0, \epsilon)_{\trho}$.

We now write $\Vec'$ on the space $[\X; Z]$. Notice that $\rho \partial_\rho + \mu \partial_\mu + \mu' \partial_{\mu'}$ is precisely $\trho \partial_{\trho}$. 
Also, we can easily check that $\partial_{y_j}$, $\partial_\lambda$
and $\trho \partial_\mu$ lift to  smooth vector fields on $[\X;
Z]$. Using the assumption that $V \in \rho^{\alpha} C^{\infty,
  \underline{\epsilon}}(\overline{\mathbb{R}^d}) \subset (\trho
\rho_B)^{\alpha} C^{\infty, \underline{\epsilon}}([\X ; \tilde Z])$,
where we use the notation of Section~\ref{subsec:conormal}, 
we compute that $\Vec'$ takes the form
\begin{equation}\begin{gathered}
\Vec' = \Big( -2\lambda \trho + O(\rho_B^\alpha \trho^\alpha C^{\infty, \underline{\epsilon}}) \Big) \partial_{\trho} + 
O(\rho_B^{\alpha+1} \trho^{\alpha-1} C^{\infty, \underline{\epsilon}} ) \partial_{\rho_B} \\
+ h^{ij}\mu_i \partial_{y_j} + \Big(2 h^{ij} \mu_i
\mu_j + \rho \frac{\partial V}{\partial \rho} \Big)
\partial_\lambda + O(\rho_z C^\infty + \rho_B^\alpha \trho^{\alpha-1} C^{\infty, \underline{\epsilon}}) \partial_{\theta},
\end{gathered}\end{equation}
where the conormal coefficients of $\p_{\trho}$ and $\p_{\rho_B}$ come
from the $\p_{y_k} V$ coefficients of $\p_{\mu_k}$ in \eqref{Vec'}.
Note that $\Vec'$ is a `conormal b-vector field,' meaning it has
conormal regularity and is
tangent to both boundary hypersurfaces $B$ and $\wt{Z}$; this can be
seen directly by noting that all the $\p_{\rho_B}$, resp.\ $\p_{\trho}$
terms vanish at $B$, resp.\ $\wt{Z}$.

We will multiply $\Vec'$ by a function so that near $\tilde Z$ we can
use $\trho$ as a parameter for the flow.  Thus for $0 < c < 1 / 2$ to be
chosen below, denote by
$\kappa$ the function on $\SR$ equal to $1$ for $|\lambda| \leq c \sqrt{E}$, and
equal to $-2(\sgn \lambda)\sqrt{E} \trho / (\Vec' \trho)$ for $|\lambda| \geq (1-c) \sqrt{E}$. 
Letting $\Vec'' = \kappa \trho^{-1} \Vec'$, we have that
$\Vec''(\trho) = -2(\sgn \lambda)\sqrt{E}$ near $\cup_{\pm} \p_{\pm}
\SR$, and it follows that for small enough $c$,
$\Vec''$ is a smooth vector field on the interior of the blown up space taking the form
\begin{equation}\begin{gathered}
\Vec'' =  -2(\sgn \lambda)\sqrt{E} \partial_{\trho} + 
O(\rho_B^{\alpha} \trho^{\alpha-2} C^{\infty, \underline{\epsilon}} ) \rho_B \partial_{\rho_B} \\
 + O(C^\infty + \rho_B^\alpha \trho^{\alpha-2} C^{\infty, \underline{\epsilon}}) \partial_{y, \lambda, \theta},
\end{gathered}\label{eq:time-and-rho}\end{equation} 
in the region $|\lambda| > (1 - c) \sqrt{E}$, while the coefficient of
$\p_\trho$ lies in $C^\infty + \rho_B^\alpha C^{\infty, \epsilon}$
outside this region.
Notice that $\kappa =1$ at the boundary of $\SR$, and that
$\Vec''$ is tangent to $B$, but transverse to $\tilde Z$, pointing
`inward' for $\lambda < 0$ and `outward' for $\lambda > 0$. We also
note for future reference, that, if $(\Vec^0)''$ is the corresponding
vector field for the zero potential, that 
\begin{equation}\begin{gathered}
\Vec'' - (\Vec^0)'' =  O(\rho_B^\alpha \trho^{\alpha-1} C^{\infty, \underline{\epsilon}})\partial_{\trho} + 
O(\rho_B^{\alpha} \trho^{\alpha-2} C^{\infty, \underline{\epsilon}} )  \rho_B \partial_{\rho_B} \\
+ O(\rho_B^\alpha \trho^{\alpha-2} C^{\infty, \underline{\epsilon}}) \partial_{y, \lambda, \theta}.
\end{gathered}\label{VecVec0}\end{equation}

 \begin{definition} We define the Lagrangian submanifold $\SR$ as follows: 
 we start from the `initial condition'
\begin{equation}
\partial_- \SR := \{ \trho = 0, y = y', \eta = -\eta', \lambda = -\sqrt{E} \} \subset [\X; Z],
 \label{SR-}\end{equation}
written using the coordinates \eqref{coords1}, which is a submanifold of $\tilde Z$.\footnote{Near the boundary of $\tilde Z$, we use the coordinates \eqref{coords2} and
write it in the form 
\begin{equation}
 \{ |\mu| = 0, \ - \hat \mu' = \hat \mu, \ |\mu'|/|\mu| = 1, \  \omega = \omega',  \ \lambda = -\sqrt{E}, \ \phi = 0 \}.
 \end{equation}
This is clearly a smooth submanifold of $\tilde Z$.} Then $\SR$ is defined as
the flow out from $\partial_- \SR$  using the vector field $\Vec''$,
that is, the union of all integral curves of $\Vec''$ starting at
points of $\partial_- \SR$. 
\end{definition}

\begin{lemma}\label{lem:uft}
All integral curves of $\Vec''$ starting at $\partial_- \SR$ reach the set $\tilde Z \cap \{ \lambda = +\sqrt{E} \}$ in finite time. 
\end{lemma}

 \begin{definition}\label{thm:SRplus}
 We define $\partial_+ \SR$ to be the intersection of $\SR$ with $\tilde Z \cap \{ \lambda = \sqrt{E} \}$. 
\end{definition}

\begin{remark} In \cite{HW2008} the sojourn relation was described as having conic singularities at the outgoing radial set $G^\sharp$, which were resolved by blowing up the span of this set. This blowup corresponds to the blowup of $Z$ already performed here. 
\end{remark}

\begin{proof}[Proof of Lemma~\ref{lem:uft}] 
 Notice that $\partial_- \SR$ is contained in the energy surface $\{ \lambda^2 + h^{ij} \mu_i \mu_j + V = E \}$. By conservation of energy, the integral curves 
starting from $\partial_- \SR$ are completely contained in this energy surface. 
Noting that $\mu = 0$ and $V = 0$ at $\tilde Z$, the integral curves
can only meet $\tilde Z$ at $\lambda = \pm \sqrt{E}$.

We first show that trajectories contained in the original
boundary hypersurface $B$ return to $\tilde Z \cap \{ \lambda = +\sqrt{E} \}$ in finite time. 
In this region, since $\rho / |\mu|$ and $|\mu'|/|\mu|$ are bounded,
we can take the boundary
defining function for $\tilde {Z}$  to be $\trho = |\mu|$.  Then consider the  vector field $\trho^{-1} \Vec'$, which is the same as $\Vec''$ up to reparametrization, and hence has the same integral curves. We compute that inside the boundary hypersurface $B$, the variables $\lambda$ and $|\mu|$ satisfy  
\begin{equation}\begin{gathered}
\dot \lambda = |\mu|, \quad \dot{| \mu|} = -\lambda, \quad \dot {|\mu'|} = - \lambda, \quad |\mu| = \sqrt{h^{ij} \mu_i \mu_j}. 
\end{gathered}\end{equation}
This has an exact solution $|\mu| = |\mu'| =  \sqrt{E}\sin s$, $\lambda = -
\sqrt{E}\cos s$ where $s \in [0, \pi]$ is the `time' parameter along this
reparametrized bicharacteristic and $\exp$ is the exponential map on
the sphere.\footnote{As this is happening, $y'$ traces out a geodesic on $\mathbb{S}^{d-1}$, of length $\pi$, i.e. half a great circle.} In particular, it returns to $\tilde
Z$ in finite time, at $\lambda = +\sqrt{E}$, as claimed. Then by continuity, nearby trajectories also reach $\tilde
Z$ in finite time. As observed above, this can only be at $\lambda = \pm\sqrt{E}$ and by continuity, it must be at $\lambda = +\sqrt{E}$. 

Now consider integral curves starting at $\partial_- \SR$ and in the interior of $\tilde Z$. These integral curves immediate pass into the interior of $[ \X; Z]$, i.e. into $\{ \rho > 0 \}$. By the nontrapping hypothesis, they return to $\{ \rho = 0 \}$, and this can only be at $\tilde Z$, as the vector field $\Vec''$ is tangent to $B$. Since $\Vec''$ is inward pointing at $\tilde Z$ for $\lambda <0$ and outward pointing for $\lambda > 0$, according to \eqref{eq:time-and-rho}, this must occur at $\lambda > 0$, hence at $\lambda = +\sqrt{E}$. 
 \end{proof}

Note that the interior of $\tilde Z \cap \{ \lambda = \sqrt{E} \}$ can be identified  with
 $T^*\mathbb{S}^{d-1} \times T^*\mathbb{S}^{d-1}$; indeed it as
 discussed above, $(y, \eta, y', \eta')$ give smooth functions on the
 interior of $\trho = 0$, and thus $\tilde Z \cap \{ \lambda =
 \sqrt{E} \}$ inherits a symplectic structure.  The interior of the set $\p_- \SR$ is
 thus identified with $T^* \mathbb{S}^{d-1}$ as the diagonal, and
 $\p_- \SR$ itself is in fact the ball bundle obtained by
 radially compactifying the fibers of $T^* \mathbb{S}^{d-1}$.  The boundary $\partial_+ \SR$ of $\SR$
 restricts to be a Lagrangian submanifold of this space,
 and coincides with the graph of the reduced scattering map
 $\mathcal{S}_E$ at energy $E$. Thus, the submanifold $\p_+ \SR$
 \textit{is precisely the
Lagrangian for the `absolute scattering matrix'.}
The relative
scattering matrix $S_h$, which is the object we are studying in this
paper, is the composition of the absolute scattering matrix with the
antipodal map multiplied by $i^{(d-1)/2}$; this normalization ensures
that the scattering matrix for the zero potential is the identity.

Note further that the integral curves of $\Vec''$ have initial
condition on the manifold with boundary $\p_- \SR$, that the boundary
of $\p_- \SR$ can be defined by the restriction of $\rho_B$ to $\p_-
\SR$ (as its boundary is exactly its intersection with $B$), and that
one expects integral curves $\gamma_p(\tau)$, where $p \in \p_- \SR$
is the initial value, to not be smooth in three places: 1) at $\tau =
0$, 2) at $p \in B$, and 3) when $\tau = T_p,$  the \textit{exit
  time}, i.e.\ the time when $\gamma_p$ intersects $\p_+ \SR$.  

We can now prove Proposition~\ref{thm:sojourn-near-infty}.
\begin{proof}[Proof of Proposition~\ref{thm:sojourn-near-infty}]
Let $T(y', \eta')$ be the time (in terms of the vector field $\Vec''$) taken to reach $\partial_+ \SR$ starting at $(y', \eta') \in \partial_- \SR$. Also, let $Y(y', \eta', \tau)$ and $N(y', \eta', \tau)$ be the solutions of the ODE \eqref{eq:time-and-rho} for $y$, respectively $\eta$. For large $|\eta|$ we use inverse polar coordinates $\hat \eta, |\eta|^{-1}$ and write $\hat N$ and $|N|^{-1}$ for the corresponding ODE solutions. 
 Thus,  the map $\mathcal{S}$ can be expressed in the form
\begin{equation}
\mathcal{S}(y', \eta') = \big( Y(y', \eta', T(y', \eta')), N (y', \eta', T(y', \eta')) \Big).
\label{SYN}\end{equation}

We first prove the following claim:
\begin{multline}
\text{For $|\eta'| \leq R < \infty$, $Y(y', \eta', T(y', \eta'))$ and $N (y', \eta', T(y', \eta'))$ are $C^\infty$ functions of $(y', \eta')$.} \\ \text{  For large $|\eta'|$, $(Y, \hat N)$ are $C^\infty + |\eta'|^{-\alpha} \Cinfe$ functions of $(y', \hat \eta', |\eta'|^{-1})$,} \\ \text{  while $|N|^{-1}$ is $|\eta|^{-1}$ times a $C^\infty + |\eta'|^{-\alpha} \Cinfe$ function of $(y', \hat \eta', |\eta'|^{-1})$.}
\end{multline}
To prove this for $|\eta'| \leq R$, we choose a small $\delta > 0$ and write $T'(y', \eta') = T(y', \eta') - \delta$. Because of the form of $\Vec''$ near $\partial_+ \SR$, this is the time taken for the trajectory starting at $(y', \eta') \in \partial_- \SR$ to reach the set $\{ \trho = \delta, \lambda > 0 \}$. As a consequence of Proposition~\ref{prop:conormal-regularity}, we see that $Y(y', \eta', T'(y', \eta'))$ and $N (y', \eta', T'(y', \eta'))$ are $C^\infty$ functions of $(y', \eta')$. (Unfortunately, we cannot immediately make the same claim with $T$ replacing $T'$, because $Y(y', \eta', \tau)$ fails to be smooth in $\tau$ precisely at $\tau = T$.) Now define the map 
$$
(y_0, \eta_0) \mapsto S_\delta(y_0, \eta_0),
$$
where $\gamma^{-1}_{y_0, \eta_0}$ is the trajectory that meets $\partial_+ \SR$ at $(y_0, \eta_0)$ and $S_\delta(y_0, \eta_0)$ are the $(y, \eta)$ coordinates of the intersection of  $\gamma^{-1}_{y_0, \eta_0}$ with $\{ \trho = \delta, \lambda > 0 \}$. Again using Proposition~\ref{prop:conormal-regularity}, we see that $S_\delta$ is a smooth map. Moreover, since $\Vec''$ is Lipschitz, the map $S_\delta$ is invertible for  $\delta$ sufficiently small. 
From these observations, and  \eqref{SYN},  we see that 
$$
\mathcal{S}(y', \eta') = S_\delta^{-1}\Big( Y(y', \eta', T'(y', \eta')), N (y', \eta', T'(y', \eta')) \Big) 
$$
 is smooth. 
 
 To prove the claim for $|\eta'|$ large, we follow exactly the same steps, replacing $C^\infty$ regularity by $C^\infty + |\eta'|^{-\alpha} \Cinfe$ regularity in terms of the boundary defining function $|\eta'|^{-1}$ for $\partial_- \SR$, making use of Lemma~\ref{lem:inverses}.

Now we consider the effect of the potential $V$ (compared to the zero
potential) on these functions. Let $A$ be the antipodal map on the sphere, and $A^*$ the induced map on its cotangent bundle.   In the case of zero potential, at
$\partial_+ \SR$, $A^*(y, \eta)$ is equal to $(y', \eta')$. Since the
potential $V$ has the effect of perturbing $\Vec''$ by an
$O(\rho_B^{\alpha})$ term, $\rho_B = |\eta'|^{-1}$, we see from Lemma~\ref{lem:comparison} that
$y$ is given by $A(y')$, where $A$ is the antipodal map, plus a
$|\eta'|^{-\alpha} \Cinfe$ function
of the initial values $(y', \hat{\eta'}, 1/|\eta'|)$. Similarly, after
applying $A^*$, $1/|\eta|$ is equal to $1/|\eta'|$ plus a 
$|\eta'|^{-\alpha-1} \Cinfe$ function of $(y', \hat{\eta'}, 1/|\eta'|)$. 
If we write these statements in terms of the
Euclidean variables $\eta$ and $\eta'$, they translate precisely into
\eqref{conreg}.

 The function $\varphi$ is discussed in
 \eqref{eq:varphi}--~\eqref{varphireg} below.
  
\end{proof}

\subsection{Semiclassical parametrix for the Poisson operator}\label{sec:parametrix}
Heuristically speaking, e.g.\ from
\eqref{eq:generalized-eigenfunction}, the Scattering matrix is the
limit of the \textit{incoming} Poisson operator to the sphere at infinity after suitably
rescaling, composing with the antipodal map, and localizing in
frequency so as to extract only the outgoing part.  We make a more precise statement now followed by a
characterization of the Schwartz kernels of both the Poisson operator
and scattering matrices.

 The Schwartz kernel of the scattering matrix, as a half-density, is
the distributional limit of
$$
A e^{ \frac 14 \pi i (d -1) }r^{-1/2} e^{-ir\sqrt{E}/h} M_{out} P_h(r,
\omega, \omega') \rvert_{r = \infty},
$$
where $P_h = P_h(E)$ is
the  incoming Poisson operator, $A$ is the antipodal map, and
$M_{out}$ is a cutoff to semiclassically outgoing frequencies. This requires some
explanation, which we give a rough version of now with details to follow.  
First of all, here we are regarding $P_h$ has a half-density, by multiplying by
$|dx d\omega'|^{1/2} = |r^{d-1} dr d\omega d\omega'|^{1/2}$ on
$\mathbb{R}^d_x \times \mathbb{S}_{\omega'}^{d-1}$. For example,
$P_{h, 0}(E)$, the  incoming Poisson operator for the zero potential is 
$$
P_{h, 0}(E) = (\sqrt{E}/2\pi h)^{(d-1)/2} e^{- ix \cdot \omega' \sqrt{E}/h} |dx d\omega'|^{1/2}.
$$
By \cite[Eqn 1.13]{GST}, for $\psi \in C^\infty(\mathbb{S}^{d-1})$,
letting $A$ denote the antipodal map of $\mathbb{S}^{d-1}$, we
have
$$
P_{h, 0}(E) (\psi |d\omega|^{1/2}) \sim r^{-(d-1)/2} ( e^{-i r
  \sqrt{E}/h} e^{\frac 14 \pi i (d -1)} \psi(\omega) + e^{i r
  \sqrt{E}/h} e^{- \frac 14 \pi i (d -1)} A^*\psi) |dx|^{1/2},
$$ 
and for general (decaying, smooth) potentials, $P_h(E) (\psi
|d\omega|^{1/2})$ satisfies the same expression with $A^* \psi$
replaced by $A^* S_h \psi$, where $S_h = S_h(E)$ is the scattering matrix.  Taking into account the half density
factor on $\mathbb{R}^d$, one then has
\begin{equation*}
  \begin{split}
    &r^{-1/2} e^{-ir \sqrt{E}/h} P_h(E) (\psi |d\omega|^{1/2}) \\
    &\qquad \qquad \sim ( e^{- 2 i r \sqrt{E}/h} e^{\frac 14 \pi i (d -1)}
    \psi(\omega) + e^{- \frac 14 \pi i (d -1)} A^* S_h(\psi)) |
    \frac{dr}{r} d\omega |^{1/2}.
  \end{split}
\end{equation*}
The half-density $|dr / r|^{1/2}$ is special; it is exactly the radial
half-density which makes sense to leading order invariantly at the
sphere at infinity of the radially compactified Euclidean space
$\overline{\mathbb{R}^d}$.  One thus wishes to cancel off the $|dr /
r|^{1/2}$ factor, to microlocalize away from the $e^{-2ir \sqrt{E}/h}$
frequency, and then take the limit $r \to \infty$.  Composing with $A e^{\frac 14
  \pi i (d -1)}$ will then give the scattering matrix.

In the case that $V$ is a smooth function viewed on $\overline{\RR^d}$, which requires in particular that $\alpha$ is an integer, the semiclassical Poisson operator was constructed in \cite{HW2008} as a sort of `boundary value' of the resolvent kernel. However, the Poisson operator can also be constructed directly. Here we make some remarks on this construction in the case that $V$ has regularity $\rho^\alpha \Cinfe$. 

We wish to construct a Fourier integral operator $F$ which is a parametrix for the Poisson operator. That is, it should have the property that $(h^2 \Delta + V - E) F_h(\phi) \in h^\infty \rho^\infty C^\infty(\overline{\RR^d})$, and also that 
$$
F_h(\phi) \sim 
 r^{-(d - 1)/2} (e^{- i \sqrt{E} r/h} \phi(\omega) + e^{i
  \sqrt{E} r/h}\psi(-\omega) ) + o(r^{(d - 1)/2}), \quad r = |x| \to \infty. 
$$
That is, up to $O(h^\infty \rho^\infty C^\infty)$ errors, $F_h \phi$ is a distorted plane wave for $\Delta + V $ of energy $E$, and has incoming boundary data $\phi$. 
Then we will have $e^{i\pi(d-1)/2}\psi = S_h(\phi)$ up to an $O(h^\infty C^\infty)$
error. 

Based on \cite{HW2008}, our ansatz is that $F_h$ is a Fourier integral operator associated to the Lagrangian submanifold $\SR$. The principal symbol $\sigma_0$ should satisfy the transport equation
$$
\mathcal{L}_{\Vec} a_0 = 0,
$$
where $a_0$ is a half-density on $\SR$. Moreover, at $\partial_- \SR$, we have an initial condition for $a_0$. This arises from the microlocally incoming condition on the plane wave $F_h \phi$, that is, the condition that the incoming boundary data be $\phi$. Microlocally this translates to the condition that $\rho^{1/2} a_0$ restricts to $\partial_- \SR$ to be the canonical half-density $|dy' d\eta'|^{1/2}$ there. It is not hard to see that this implies that the half-density $a_0$ is equal to $|dy' d\eta' dt|^{1/2}$, where $t$ is the time parameter along the $\Vec$-trajectories.  

Now consider the higher order symbols in any local parametrization of the FIO. A local parametrization
takes the form
\begin{equation}
h^{-\dim v/2} \int e^{i\Psi(y, y', \rho, v)/h} \sum_{j=0}^\infty h^j
a_j(y, y', \rho, v) \, dv \times \big| \frac{dy dy' d\rho}{\rho^{d+1}}
\big|^{1/2}. \label{eq:6}
\end{equation}
Without loss of generality, we can assume that $a_j$ depends on a minimal number of variables, 
that is, $\dim \SR = 2d-1$ of the variables $(y, y', \rho, v)$. We
denote these variables collectively by $\blambda$. In that case, the symbols $\sigma_j$, $j \geq 0$, given by
$$
\sigma_j = a_j(\blambda) \Big| \frac{\partial(\blambda, d_v \Psi)}{\partial (y, y', \rho, v)} \Big|^{-1/2} |d\blambda|^{1/2},
$$
are formally determined by $\sigma_0$ and are solutions of an equation of the form
$$
\mathcal{L}_{\Vec} \sigma_j = Q \sigma_{j-1},
$$
where $Q$ is a second order operator on half-densities on $\SR$, depending on the particular variables on which $a_j$ depends. The operator $Q$ is induced by the Laplacian. Because of this it is $\rho^2$ times a b-differential operator of order 2. The regularity of the coefficients of $Q$ on $\SR$ is determined by the regularity of $\SR$, i.e. they take the form $C^\infty + \trho^{\alpha-1} \rho_B^\alpha C^\infty$. 

We can write this transport equation on $\SR$ in the coordinates $(y', \eta', t)$. Writing $\sigma_j = s_j |dy'd\eta' dt|^{1/2}$, it implies the ODE 
$$
\frac{\partial}{\partial t} s_j + k s_j = \rho^2 \tilde Q s_{j-1},
$$
where $\tilde Q$ is a scalar second order b-differential operator and, near $\partial_{\pm} \SR$, we have $k = 2\rho^2 \partial_\rho \Lambda(y', \eta', \rho)$. Now changing variable from $t$ to $\trho$, and dividing by a factor of $\rho$, we have near $\partial_{\pm} \SR$, 
\begin{equation}
2 \Lambda \trho \partial_{\trho} s_j = \big( 2 \rho \partial_\rho \Lambda(y', \eta', \rho) \big) s_j + \rho \tilde Q s_{j-1}.
\label{transeqn}\end{equation}
Moreover, we have an initial condition $s_j = 0$ at $\partial_- \SR$. 
Propositions~\ref{prop:conormal-regularity} and \ref{prop:cr2} apply
to this ODE and show that $s_j$ has the regularity $C^\infty +
\trho^{\alpha-1} \rho_B^\alpha C^\infty$ away from $\partial_+ \SR$,
and has the regularity given by part (iii) of
Proposition~\ref{prop:cr2} near $\partial_+ \SR$. In fact we can
conclude more regarding the vanishing of the $s_j$ using the ODE
comparison lemma, Lemma \ref{lem:comparison} above.  Indeed, as the
coordinates $(y', \eta', t)$ provide coordinates on both $\SR$ and
$\SR^0$, the free scattering relation, we can compare the $s_j$ with
the $s_j^0$, the functions arising analogously in the free case, which
satisfy that $s_j^0 \equiv 0$ for $j \ge 1$ and $s^0_0 \equiv 1$.  It is
straightforward to check that the ODEs for the $s_j^0$ differ in an
$\trho^{\alpha - 1}\rho_B^\alpha  C^{\infty, \epsilon}$ manner, and
thus the lemma implies that the $s_j$ for $j \ge 1 $ and $1 - s_0$ lie
in $\trho^{\alpha-1} \rho_B^\alpha C^\infty$.\footnote{Further analysis shows that in fact $s_j$ vanishes to
  order $\alpha + j$ for $j \geq 1$, but we do not need this fact
  here.}

Using these symbols we can build an FIO parametrix for the Poisson
kernel, that solves the equation $(\Delta + V - E) F_h \phi = 0$ up to
an error in $h^\infty e^{i\sqrt{E} r/h} r^{-(d+1)/2}
C^\infty(\overline{\RR^d})$; an additional step reduces the error to
$h^\infty \rho^\infty C^\infty(\overline{\RR^d})$\footnote{This is
  just as in the parametrix construction of Melrose-Zworski}. The
error term can be solved away by applying the outgoing resolvent, and
this contributes a correction term to the parametrix of the form
$h^\infty e^{ir/h} r^{-(d-1)/2} C^\infty(\omega, \omega')$. This follows from
\cite[Section 12]{M1994} (showing smoothness in the $y, y'$ variables)
and \cite{VZ2000} (showing that the correction is $O(h^\infty)$). 

Therefore
this contributes a smooth term that is $O(h^\infty)$ to the scattering
matrix, and this has no effect on the conclusion of
Lemma~\ref{lem:S-structure}. So, in the proof in the next subsection,
it suffices to analyze the parametrix for the Poisson operator.

\subsection{The scattering matrix: proof of Lemma~\ref{lem:S-structure}}\label{sec:scattering-matrix-deduction}
The scattering matrix is obtained by taking a distributional limit of
the Poisson kernel as $r \to \infty$. We can break up the Poisson
operator microlocally into a piece microsupported away from
$\partial_+ \SR$, and a piece microsupported away from $\partial_-
\SR$. For the first piece, localized away from $\partial_+ \SR$, if we
multiply the kernel by $\rho^{1/2} e^{i\sqrt{E}/(h\rho)}$ and then take
the canonical restriction to $\rho = 0$, we obtain the identity
operator; that is, if we let this piece of $P_h(E)$ operate on a
smooth function $f(\omega')$, then multiply the result by $\rho^{1/2}
e^{i\sqrt{E}/(h\rho)}$ and then take the canonical restriction to $\rho =
0$, we obtain $f$. 
For the second piece, localized away from $\partial_- \SR$, if we
multiply the kernel by $\rho^{1/2} e^{-i\sqrt{E}/(h\rho)}$ and then take
the canonical restriction to $\rho = 0$, we obtain the scattering
matrix.

We now see how this happens at the level of kernels. We need to analyze a microlocal representation of $F_h$  more carefully near the corner of $\SR$ at the intersection of $\partial_+ \SR$ and $\rho_B = 0$. 
In this region, the Poisson operator can be expressed as an oscillatory integral using the results of \cite{HW2008}.  After a rotation of coordinates, we can assume that  $\mu_1$ and $1/\eta_1$ furnish local boundary defining functions $\trho$ and $\rho_B$, respectively.  We can then use coordinates 
\begin{equation}
\mathcal{Z} = (y', \mu_1, 1/\eta_1, \eta_j/\eta_1), \quad 2 \leq j \leq d-1, 
\end{equation}
 on $\SR$ near a point at $B \cap \tilde Z$, that is at $\rho_B = \trho = 0$. In terms of these we can write the other coordinates as smooth functions:
\begin{equation}
y_j = Y_j(\Z), \quad \Phi = \Phi(\Z). 
\end{equation}
This $\Phi$ is the same as that in the previous paragraph thought of
as a function on $\SR$.
In the expressions below we will replace $\mu_1$ by $\sigma$ and $\eta_j/\eta_1$ by $v = (v_2, \dots, v_{d-1})$. 
Then, according to \cite[Section 6.3]{HW2008}, $\SR$ has a local parametrization of the form 
\begin{equation}
\Psi(r, y, y', \sigma, v) = \Phi(y', \sigma, \frac{\rho}{\sigma}, v) + \frac{\sigma}{\rho} \big(y_1 - Y_1(y', \sigma, \frac{\rho}{\sigma}, v)\big) + \sum_{j=2}^{d-1}  \frac{\sigma}{\rho} v_j \big(y_j - Y_j(y', \sigma, \frac{\rho}{\sigma}, v)\big),
\label{Poisson-param}\end{equation}
Then a microlocal parametrix for the  
Poisson operator $P_h(E)$ takes the form
  \begin{equation}
    (\sqrt{E}/(2\pi h))^{-(d-1)/2} \int e^{i\Psi/h}  \rho^{-(d-1)/2} \sigma^{d-2} a(\sigma, \frac{\rho}{\sigma}, y', v, h) \, d\sigma \, dv |\frac{d\rho dy}{\rho^{d+1}} dy'|^{1/2},
    \label{Poisson-local}
  \end{equation}
  where $\Psi$ is as in \eqref{Poisson-param}\footnote{The powers of
    $\rho$ and $\sigma$ are as given by \cite[Equation
    (6.21)]{HW2008}, after taking into account that the half-density
    used there is $\rho^{-(d-1)/2} h^{-(2d-1)/2}$ times the
    half-density $|dx d\omega'|^{1/2}$ used here, times
    $|dh/h^2|^{1/2}$.}
 and the amplitude $a$ has regularity as determined by the regularity of the functions $s_j$ as described above. 

If we multiply  \eqref{Poisson-local} by $\rho^{1/2} e^{-i\sqrt{E}/(\rho
  h)}$ and restrict to $\rho = 0$, using $d\sigma dv =
\rho^{d-1}\sigma^{-(d-2)} d\eta$ we obtain the following oscillatory
integral expression, with the given subsitution we obtain
\begin{equation}
  \label{eq:Smatrix-f}
  \begin{split}
&   \int e^{i\Psi/h} \rho^{-(d-1)/2}
    \sigma^{d-2} a(\sigma, y', \frac{\rho}{\sigma},  v, h) \, d\sigma
    \, dv |\frac{d\rho dy}{\rho^{d+1}} dy'|^{1/2} \\
&\qquad =  (\int e^{i\Psi/h} \rho^{(d-1)/2} a(\sigma, y', \eta, h) \, d\eta)
|\frac{d\rho dy}{\rho^{d + 1}}  dy'|^{1/2}.
  \end{split}
\end{equation}
Regarding $\Phi$ as a function on $\SR$, its behaviour as we approach
$\partial_+ \SR$ was worked out in \cite[Section
2]{Gell-Redman-Hassell-Zelditch} (where the function is called
$\phi_0$). It is shown that 
$$
\Phi = \sqrt{E} r + \varphi + o(1), \trho \to 0,
$$
where $\varphi$ is a function on $\SR$ constant on trajectories, given by 
\begin{equation}
  \label{eq:varphi}
  \varphi(y', \hat \eta', \rho_B)  = \int_{\gamma} x \cdot \nabla V,
  \  \mbox{ for $\gamma$ the trajectory with initial
    condition $(y', \hat \eta', \rho_B)$}.
\end{equation}
Thus $\varphi$ is the boundary value of a function $\tilde{\varphi}$ on $\SR$ satisfying the ODE
$
\Vec(\tilde\varphi) = x \cdot \nabla V 
$
with a zero initial condition at $\partial_+ \SR$. Using Propositions~\ref{prop:conormal-regularity}, \ref{prop:cr2} and Corollary~\ref{cor:solat0}, we see that 
\begin{equation}
\varphi \in \rho_B^{\alpha-1} \Cinfe(\partial_+ \SR).
\label{varphireg}\end{equation} 
(This is the final part of the proof of Lemma \ref{thm:sojourn-near-infty}.)

Multiplying by $e^{ \frac 14 \pi i (d -1) }\rho^{1/2}
e^{-i\sqrt{E}/(\rho h)}$
and taking the distributional limit at $\rho = 0$ then gives
\begin{equation}
(\sqrt{E}/(2\pi h))^{-(d-1)/2} \int e^{i \big( \varphi(y', \eta) + \sum_j(y_j - Y_j(y', \eta)) \eta_j  \big)     /h}  a(0, y', \eta, h) \, d\eta |dy dy'|^{1/2}.
\label{Smatrix}
\end{equation}
We write the phase in this equation as $(y-A(y')) \cdot \eta + G(y', \eta)$. 
From \eqref{varphireg}, as well as $Y = A(y') + O(\rho^\alpha \Cinfe)$ from \eqref{conreg}, we see that 
$G \in \rho^{\alpha-1} \Cinfe$.

Then, from the fact that the principal symbol of the scattering matrix
is $|dy d\eta|^{1/2}$ (see \cite[Lemma 3.1]{DGHH2013}), we find that
$$
a(0, y', \eta, 0) = \Big| \det \frac{\partial(y, \eta, y-A(y') + d_\eta G)}{\partial(y', y, \eta)} \Big|^{1/2} = \big| \det (\Id + d^2_{y' \eta} G) \big|^{1/2} .
$$
This shows that 
$$
a(0, y', \eta, 0) = 1 + O(\rho^\alpha \Cinfe).
$$
Moreover, it follows from the construction of the functions $a_j$ in
\eqref{eq:6} that each term in the expansion of $a$ in $h$ as $h \to
0$ is bounded by $O(\rho^\alpha \Cinfe)$ and also the error
contributes to lower order that $a - 1 \in S^{-\alpha}$.  Composing with
the antipodal map gives \eqref{eq:Sh-at-fiber-infinity} and thus Lemma \ref{lem:S-structure} follows.

\section{Powers of the scattering matrix}\label{sec:composition}

In this section we will prove Lemma \ref{thm:composition}, which gives
an expression for powers of the scattering matrix $S_h^k$ near `fiber
infinity'.  Indeed, recall that, notation as in Lemma \ref{thm:composition},
$S_h = F_1 + F_2$ where $F_1$ is a semiclassical FIO with compact
microsupport and $F_2$ (or rather its Schwartz kernel) is given by the
oscillatory kernel expression in \eqref{eq:Sh-at-fiber-infinity}, and as
discussed in the proof of Lemma
\ref{thm:trace-formula}, the only non-compactly microlocally supported
term in $S_h^k$ is $F_2^k$.

Recall the discussion in Section \ref{sec:traces-and-compositions}
(see near \eqref{eq:half-densities})
explaining that the oscillatory integral giving $F_2$ is to be thought
of as a half-density on $\mathbb{S}^{d-1} \times \mathbb{S}^{d-1}$.
Thus two oscillatory integrals
\begin{equation}\label{eq:oscillatory-integral}
I_i(y, y') = (2 \pi h)^{-(d -1)}\int e^{\Phi_i(\sphvar, \sphvar', \dspharv)/h} a_i(\sphvar, \sphvar',
\dspharv, h) \, d\dspharv
\end{equation}
with $i = 1, 2$ define Schwartz kernels $I_i(y, y') |dy dy'|^{1/2}$
whose composition as operators is given by
$$
I_1 \circ I_2 = (\int I_1(y, y'') I_2(y'', y') |dy''|) |dy dy'|^{1/2}.
$$
Since this just amounts to integration in the $y''$ variable in all
the expressions below we will drop the half density factors.

\begin{proof}[Proof of Lemma \ref{thm:composition}]
Suppose that $G_1(\sphvar, \dspharv)$ and $G_2(\sphvar, \dspharv)$ are symbols
of order $-\beta$ for $\beta > 0$. We will show that the
composition of two FIOs, with phase function 
\begin{equation}\label{eq:Phi-appendix}
\Phi_i(\sphvar, \sphvar', \dspharv)  := (\sphvar - \sphvar') \cdot
\dspharv + G_i(\sphvar', \dspharv)
\end{equation}
is an FIO with phase function $(\sphvar -
\sphvar') \cdot \dspharv + G_1(\sphvar', \dspharv) + G_2(\sphvar', \dspharv) +
E(\sphvar', \dspharv)$, where $E$ is a symbol of order $-2\beta$.
Thus let $I_1, I_2$ be oscillatory integrals as in
\eqref{eq:oscillatory-integral}, with $\Phi_i$ as above and 
with amplitudes $a_i \in S^{m_i}$.  
The composition has a representation of the form 
\begin{equation}\label{eq:initial-composition}
I_1 \circ I_2 : = (2\pi h)^{-2(d -1)} \int e^{\frac{i}{h}\Phi(\sphvar, \sphvar'', \sphvar', \dspharv, \dspharv')} a_1(\sphvar, \sphvar'', \dspharv, h)  a_2(\sphvar'', \sphvar', \dspharv', h) \, d\dspharv  \, d\dspharv' \, d\sphvar'',
\end{equation}
where
$$
\Phi(\sphvar, \sphvar'', \sphvar', \dspharv, \dspharv') = (\sphvar - \sphvar'') \cdot \dspharv + G_1(\sphvar'', \dspharv) +(\sphvar'' - \sphvar') \cdot \dspharv' + G_2(\sphvar'', \dspharv').  
$$
We eliminate the variables $(\sphvar'', \dspharv')$ up to an $O(h^\infty)$
error, by replacing them with their stationary values and applying the
stationary phase lemma \cite{Hvol1}. This works as the Hessian of $\Phi$ with respect to $(\sphvar'', \dspharv')$ is 
$$
\begin{pmatrix}
0 & \Id \\
\Id & 0 
\end{pmatrix}
+ O(|\dspharv'|^{-1-\alpha}),
$$
and is therefore invertible, with uniformly bounded inverse, for large $|\dspharv'|$. 
 The stationary points in $y'', \dspharv'$, i.e.\ the points where $D_{\sphvar'', \dspharv'}
 \Phi = 0$, occur at
 \begin{equation}\begin{aligned}
 \sphvar'' &= \sphvar' - d_{\dspharv'} H(y'', \eta', \eta), \\
 \dspharv' &= \dspharv + d_{\sphvar''} H(y'', \eta', \eta),
 \end{aligned}\end{equation}
where
$$
H(y'', \eta', \eta) = G_1(y'', \eta) + G_2(y'', \eta').
$$
The second line in the above equation array shows that on the critial
set we can write $\eta = \eta(\eta', y'')$ with $\eta - \eta' \in S^{1
  - \beta}$.  Thus, we want to invert the transformation 
 \begin{equation}
 \begin{pmatrix} \sphvar' \\ \dspharv \end{pmatrix} 
 =  \begin{pmatrix} \sphvar'' \\ \dspharv' \end{pmatrix} + 
 \begin{pmatrix} - d_{\dspharv'} H(y'', \eta') \\ d_{\sphvar''} H(y'', \eta') \end{pmatrix}
 \label{invert-trans}\end{equation}
 when $|\dspharv|$ is large. It is easy to see, by the method of successive approximations for example, that the inverse exists for large $\dspharv$. We claim that the inverse map,
 which we write in the form $\sphvar''(\sphvar', \dspharv), \dspharv'(\sphvar', \dspharv)$,  is the identity plus a symbol of order $-\beta$. To see this, we differentiate \eqref{invert-trans} with respect to $\sphvar'$ and $\dspharv$ to obtain 
 \begin{equation}
 \begin{pmatrix} \Id & 0  \\ 0 & \Id  \end{pmatrix}
= 
 \Bigg( 
 \begin{pmatrix} \Id & 0  \\ 0 & \Id  \end{pmatrix} + 
 \begin{pmatrix}
 - d^2_{\dspharv'\sphvar''} H(\sphvar'', \dspharv') & d^2_{\dspharv'\dspharv'} H(\sphvar'', \dspharv') \\ d^2_{\sphvar''\sphvar''} H(\sphvar'', \dspharv') & 
d^2_{\sphvar''\dspharv'} H(\sphvar'', \dspharv')
\end{pmatrix} \Bigg) 
\begin{pmatrix} \frac{\partial \sphvar''}{\partial \sphvar'} & \frac{\partial \sphvar''}{\partial \dspharv} \\ 
\frac{\partial \dspharv'}{\partial \sphvar'} & \frac{\partial \dspharv'}{\partial \dspharv} \end{pmatrix} .
\label{matrix-relation}\end{equation}
This shows that 
$$
\begin{pmatrix} \frac{\partial \sphvar''}{\partial \sphvar'} & \frac{\partial \sphvar''}{\partial \dspharv} \\ 
\frac{\partial \dspharv'}{\partial \sphvar'} & \frac{\partial \dspharv'}{\partial \dspharv} \end{pmatrix} = 
 \begin{pmatrix} \Id & 0  \\ 0 & \Id  \end{pmatrix} + 
 \begin{pmatrix}
 e_{11}(\sphvar', \dspharv) & e_{12}(\sphvar', \dspharv) \\
  e_{21}(\sphvar', \dspharv) & e_{22}(\sphvar', \dspharv)
  \end{pmatrix}
  $$
where repeated differentiation of \eqref{matrix-relation} shows that
$$
 \begin{pmatrix}
 e_{11}(\sphvar', \dspharv) & e_{12}(\sphvar', \dspharv) \\
  e_{21}(\sphvar', \dspharv) & e_{22}(\sphvar', \dspharv)
  \end{pmatrix} \in  \begin{pmatrix}
S^{-1 -\beta} & S^{-2 -\beta} \\
 S^{ -\beta} & S^{-1 -\beta}
  \end{pmatrix}.
$$
This proves that we can write
 \begin{equation}
 \begin{pmatrix} \sphvar'' \\ \dspharv' \end{pmatrix}
 =   \begin{pmatrix} \sphvar' \\ \dspharv \end{pmatrix}  + 
 \begin{pmatrix} f_1(\sphvar', \dspharv) \\ f_2(\sphvar', \dspharv) \end{pmatrix}, \quad
 f_1 \in S^{-1 -\beta}, f_2 \in S^{-\beta}.
 \label{inverse-trans}\end{equation}

 We now write the function $\wt{\Phi}$, the restriction of $\Phi$ to
 the critical set $\{ D_{\dspharv'} \Phi = 0 \}$,
 \begin{equation}\begin{gathered}
\wt{\Phi} = (\sphvar - \sphvar') \cdot \dspharv  + G_1(\sphvar', \dspharv) + G_2(\sphvar', \dspharv) + E'(\sphvar', \dspharv), \\
E'(\sphvar, \sphvar', \dspharv) =   (\sphvar' - \sphvar'')(\dspharv -
\dspharv') + G_2(\sphvar'', \dspharv') - G_2(\sphvar', \dspharv) +
G_1(y'', \eta) - G_1(y', \eta) 
 \end{gathered}\label{E}\end{equation}
Then writing $G_2(\sphvar'', \dspharv') - G_2(\sphvar', \dspharv)  =
G_2(\sphvar'', \dspharv') - G_2(\sphvar'', \dspharv) + G_2(\sphvar',
\dspharv) - G_2(\sphvar', \dspharv) $ as
$$
(\sphvar'' - \sphvar') \cdot \int_0^1 (d_{\sphvar}G_2)(\sphvar'' + t(\sphvar' - \sphvar''), \dspharv') \, dt \\ + (\dspharv' - \dspharv) \cdot \int_0^1 (d_{\dspharv} G_2)(\sphvar', \dspharv + t(\dspharv' - \dspharv)) \, dt,
$$
and 
$$
G_1(y'', \eta) - G_1(y', \eta)  = (\sphvar'' - \sphvar') \cdot \int_0^1 (d_{\sphvar}G_2)(\sphvar'' + t(\sphvar' - \sphvar''), \dspharv') \, dt, $$
we see from  \eqref{inverse-trans} that $E'$ is a symbol of order
$-2\beta$, so by stationary phase applied to~\eqref{eq:initial-composition},
$$
I_1 \circ I_2 = (2 \pi h)^{-(d -1)} \int e^{i \wt{\Phi}(y, y',
  \eta)/h} b(y, y', \eta) \, d\dspharv,
$$
where $b(y, y', \eta) = a_1(y, y'', \eta) a_2(y'', y', \eta')$
restricted to the the $\sphvar'', \dspharv'$ critical
set of $\Phi$, and, as is standard, $b \in S^{m_1 + m_2}$  with
principal symbol given by the product of the principal symbols of
$a_1$ and $a_2$. 

Lemma \ref{thm:composition} now follows by applying the above
results to repeated compositions of $F_2$ (for $k \ge 1$) or $F_2^*$
(for $k \leq 1$).  Indeed, note that the lemma is already
proven for $k = 1$ by \eqref{eq:Sh-at-fiber-infinity} and for $k = -1$
by \eqref{eq:Sh-at-fiber-infinity-conjugate}.  We focus on the $k \ge
1$ case as the $k \leq 1$ case is completely analogous.  Assuming by
induction that Lemma \ref{thm:composition} folds for $F_2^k$, consider
$F_2^{k + 1} = F_2^k \circ F_2$.  Thus the Schwartz kernels of these operators are
given by oscillatory intergrals $I_1, I_2$ corresponding $F_2^k$ and
$F_2$, respectively, as in~\eqref{eq:oscillatory-integral}
with $\Phi_1$ and $G_1$ corresponding to $F_2^k$ and $\Phi_2, G_2$  corresponding
to $F_2$. Thus $G_1(y',
\eta) = kG(y', \eta) + E_k(y', \eta)$ and $G_2 = G(y, \eta) + E'(y,
\eta)$, where $G$ comes from
the original phase function of $S_h$, i.e.\ it is as in
\eqref{eq:Sh-at-fiber-infinity} and $E_k, E' \in S^{1 - \alpha -
  \epsilon}$ for some $\epsilon > 0$.  Here as in the arguments
above we have used Remark \ref{thm:switch-y-yprime} to switch the
roles of $y'$ and $y$ in the perturbation term of the phase function.
Thus the above arguments imply that the composition has phase function
$\Phi = (k + 1)G(y', \eta) + E_k(y', \eta) + E'(y', \eta) + E''(y',
\eta) $ where $E''$ is a symbol of order $2(1 - \alpha)$.
Also, the amplitudes satisfy $a_1 - 1 \in S^{1 - \alpha}$ and $a_2 - 1
\in S^{1 - \alpha}$, then $b - 1 \in S^{1 - \alpha}$ as well.  Indeed,
this follows immediately from writing $y'', \eta'$ in terms of $y', \eta$ using
\eqref{inverse-trans}.\end{proof}

\end{appendix}

\end{document}